\documentclass{amsart}
\usepackage{enumerate, amscd, amssymb}

\allowdisplaybreaks
 \usepackage[%
   colorlinks,
   citecolor=red, linkcolor=blue,
   pdfpagelabels=true,
   unicode=true,
   pdfusetitle
 ]{hyperref}

\makeatletter
\newcommand{\doublewidetilde}[1]{{%
  \mathpalette\double@widetilde{#1}%
}}
\newcommand{\double@widetilde}[2]{%
  \sbox\z@{$\m@th#1\widetilde{#2}$}%
  \ht\z@=.9\ht\z@
  \widetilde{\box\z@}%
}
\makeatother

\newcommand\seq{\, = \,}
\newcommand\define{\mathrel{\ := \ }}
\newcommand\ede{\define}
\newcommand\where{\, \vert \ }
\newcommand\esssup{\operatorname{ess-sup}}

\newcommand{\CC}{\mathbb C}

\newcommand{\NN}{\mathbb N}

\newcommand{\RR}{\mathbb R}

\newcommand{\ZZ}{\mathbb Z}

\newcommand{\CI}{{\mathcal C}^{\infty}}
\newcommand{\CIc}{{\mathcal C}^{\infty}_{\text{c}}}
\newcommand{\CIb}{{\mathcal C}^{\infty}_{\text{b}}}
\newcommand\pa{\partial}

\newcommand{\ord}{\operatorname{ord}}

\newcommand{\Hom}{\operatorname{Hom}}
\newcommand{\End}{\operatorname{End}}

\newcommand{\Diff}{\operatorname{Diff}}

\newcommand{\dvol}{\,\operatorname{dvol}}
\renewcommand{\div}{\operatorname{div}}

\newcommand{\BiDiff}{bi{\mbox{-}}\!\operatorname{Diff}}

\newcommand\tr{\mathop{\mathrm{tr}}}

\newcommand{\mfka}{\mathfrak a}

\newcommand{\mfkb}{\mathfrak b}

\newcommand{\maD}{\mathcal D}

\newcommand{\maF}{\mathcal F}

\newcommand{\maU}{\mathcal U}

\newcommand{\maW}{\mathcal W}

\newcommand\vect[1]{\mathbf{#1}}

\newtheorem{theorem}{Theorem}[section]
\newtheorem{proposition}[theorem]{Proposition}
\newtheorem{corollary}[theorem]{Corollary}

\newtheorem{lemma}[theorem]{Lemma}
\newtheorem{notation}[theorem]{Notations}
\theoremstyle{definition}
\newtheorem{definition}[theorem]{Definition}
\theoremstyle{remark}
\newtheorem{remark}[theorem]{Remark}
\newtheorem{example}[theorem]{Example}

\author[M. Kohr]{Mirela Kohr} \address{Faculty of Mathematics and
  Computer Science, Babe\c{s}-Bolyai University, 1 M. Kog\u alniceanu
  Str., 400084 Cluj-Napoca, Romania} \email{mkohr@math.ubbcluj.ro}

\author[V. Nistor]{Victor Nistor} \address{Universit\'{e} de Lorraine,
  UFR MIM, Ile du Saulcy, CS 50128, 57045 METZ, France and
  Inst. Math. Romanian Acad.  PO BOX 1-764, 014700 Bucharest Romania}
\email{victor.nistor@univ-lorraine.fr}

\thanks{M.K. has been partially supported by
  AGC35124/31.10.2018. V.N. has been partially supported by
  ANR-14-CE25-0012-01.\\
Manuscripts available from \textbf{http:{\scriptsize
    //}iecl.univ-lorraine.fr{\scriptsize
    /}$\tilde{}$Victor.Nistor{\scriptsize /}}\\
AMS Subject classification (2010): 
58J32 (primary), 47L80, 58J05, 58H05, 46L87}

\begin{document}

\title[Sobolev spaces and $\nabla$-differential operators]{Sobolev
  spaces and $\nabla$-differential operators on manifolds I: basic
  properties and weighted spaces}

\begin{abstract}
We study {\em $\nabla$-Sobolev spaces} and {\em $\nabla$-differential
  operators} with coefficients in general Hermitian vector bundles on
Riemannian manifolds, stressing a coordinate free approach that uses 
connections (which are typically denoted $\nabla$). These concepts 
arise naturally
from Partial Differential Equations, including some that are formulated
on plain Euclidean domains, such as
the weighted Sobolev spaces used to study PDEs on singular
domains. We prove several basic properties of the $\nabla$-Sobolev
spaces and of the $\nabla$-differential operators on general
manifolds. For instance, we prove mapping properties for our differential
operators and independence of the $\nabla$-Sobolev spaces on the choices of
the connection $\nabla$ with respect to totally bounded
perturbations. We introduce a {\em Fr\'echet finiteness condition}
(FFC) for totally bounded vector fields, which is satisfied, for
instance, by open subsets of manifolds with bounded geometry. When
(FFC) is satisfied, we provide several equivalent definitions of our
$\nabla$-Sobolev spaces and of our $\nabla$-differential operators. We
examine in more detail the particular case of domains in the Euclidean
space, including the case of weighted Sobolev spaces. We also introduce
and study the notion of a {\em $\nabla$-bidifferential} operator (a bilinear 
version of differential operators), obtaining results similar to those obtained 
for $\nabla$-differential operators.
Bilinear differential operators are necessary for a global, geometric discussion of
variational problems. We tried to write the paper so that it is accessible to a
large audience.
\end{abstract}

\maketitle

\tableofcontents

\section{Introduction}\label{sec.1}

Even if one is interested only in Partial Differential Equations
(PDEs) on Euclidean domains, one is quickly lead to consider also
equations on vector bundles on manifolds, as noticed in many earlier works, including
\cite{AronszajnMilgram} and \cite{Browder61}. Sobolev spaces on
manifolds are useful, for instance, for the investigation of the
weighted Sobolev spaces that arise in the study of PDEs on polyhedral
domains. Vector bundles arise when one considers systems, such
as the elasticity and Stokes systems. This motivates our interest in Sobolev spaces and
differential operators with coefficients in a general Hermitian vector
bundle $E \to M$ on a Riemannian manifold $M$ without boundary.

The main goal of this paper is to provide
the first steps in an approach to the study of
Sobolev spaces and of differential operators on manifolds that is as
independent as possible from local coordinates. This requires us to
use {\em connections} $\nabla$ instead of derivatives in our
definitions. In this paper, as, for instance, in \cite{HebeyBook}, we
introduce Sobolev spaces and differential operators using iterations
$\nabla^j$ of the connection $\nabla$. These objects will thus be
called {\em $\nabla$-Sobolev spaces} and {\em $\nabla$-differential
  operators}.

We prove a few basic, elementary properties of the $\nabla$-Sobolev
spaces and $\nabla$-differential operators such as multiplication
properties, mapping and restriction properties, and the independence
of these definitions on totally bounded perturbations of the
connection $\nabla$. We also recall the connection between weighted
Sobolev spaces and the usual Sobolev spaces (but for a conformally
equivalent metric) on manifolds.

To obtain more in depth results, we consider the set $\maW_b(M)$ of
vector fields on $M$ all of whose covariant derivatives are
bounded. We then say that $M$ satisfies the {\em Fr\'echet finiteness
  condition} (FFC) if $\maW_b(M)$ is finitely generated as a Fr\'echet
module over $\CIb(M)$, the space of functions all of whose covariant
derivatives are bounded (see Definition \ref{def.FFC}). This condition
is somewhat close to $M$ being of bounded geometry and is satisfied
if, for instance, $M$ is an open subset of a manifold with bounded
geometry. If the (FFC) is satisfied, we provide several additional
equivalent definitions of our $\nabla$-Sobolev spaces and
$\nabla$-differential operators.

We note that our $\nabla$-Sobolev spaces and $\nabla$-differential
operators depend on the choice of $\nabla$. While
very many works have been devoted to the role of geometry in the study
of the properties of Sobolev spaces and differential operators, much
fewer works have been devoted to the issues arising from the {\em
  twisting of the coefficient bundle}. See however the papers on
``magnetic Sobolev spaces'' by Nguyen, Pinamonti, Squassina, and
Vecchi \cite{MSSp1} and by Iftimie, M\u{a}ntoiu, and Purice
\cite{Purice2007}. From a practical point of view, considering
operators with coefficients in a vector bundle has practical
applications, as it is a framework that is necessary for the modeling
of systems, such as the ones arising in solid or continuum mechanics
(Lam\'e -- and elasticity in general -- and Stokes and Navier-Stokes
systems and their generalizations), as well as in the study
of the Hodge-Laplacian. For instance, fluid mechanics on a
curved space-time in a relativistic setting was considered recently by
Disconzi, Ifrim, and Tataru \cite{RelFluids} (see also
\cite{KohrWendland18}).

The role of the {\em geometry of the underlying manifold} in the study
of the properties of Sobolev spaces and differential operators was
much studied and we cannot do justice to all the people who have
worked on the subject; nevertheless, let us mention a few of the most
important contributions that have influenced this paper. In an early
paper, Aronszajn and Milgram \cite{AronszajnMilgram} have studied
scalar differential operators on general Riemannian manifolds,
obtaining, in particular, adjoint and Green-type formulas. The reader
will find a lot of useful geometric background material accessible to
analysts. Browder \cite{Browder61} has also worked in the scalar case
and has studied PDEs on a class of euclidean domains that these days are
called ``manifolds with boundary and bounded geometry.'' More
recently, Sobolev spaces and differential operators were studied in
very many papers, see for instance \cite{AldanaCarron, AmannFunctSp, 
ammann.grosse:16b, Aubin76, Carron2001,  HebeyRobert,
  MMMT16, strichartz} and the references therein. The monographs by
Aubin \cite{AubinBook}, Hebey \cite{HebeyBook}, and Taylor
\cite{Taylor1} provide even more references. Recently Herbert Amann
and his collaborators has started a general program to study maximal
regularity and PDEs in general on certain singular spaces that can be
modelled by manifolds with boundary and bounded geometry, see, for
instance \cite{AmannFunctSp, AmannParab, AmannCauchy, Disconzi}. A
related program (but with a completely different motivation and mostly
devoted to elliptic theory) was pursued by the second named author
together with several collaborators, see, for instance \cite{AGN1,
  GN17, MazzucatoNistor1}. This paper fits into this program of the
second named author, but the role of the geometry in the study
of $\nabla$-Sobolev and $\nabla$-differential operators will mostly
be relegated to the second paper of this series \cite{KohrNistor2},
since it takes us too far afield from the results obtained in this
paper.

\subsection*{Contents of the paper}
In Section \ref{sec.2} we recall some basic definitions concerning global differential operators with smooth
coefficients and Sobolev spaces in the context of smooth manifolds and smooth vector bundles. Thus, we
recall the definition of global, geometric $\nabla$-Sobolev spaces and of weighted $\nabla$-Sobolev spaces
and obtain related multiplication and mapping properties, as well as other basic results. Also in this section,
we define the {\it $\nabla$-differential operators} and prove additional mapping properties. The next section is
devoted to totally bounded vector fields and to differential operators generated by covariant derivatives $\nabla_X$.
We call them {\it mixed differential operators} since they make the connection between the $\nabla$-differential
operators and the classical differential operators. In addition, we introduce the {\it Fr\'echet finiteness condition}
(FFC) for the set of $\CIb$ vector fields on $M$ as a main tool in the study of mixed differential operators.
We show that if such a condition is satisfied, then
the mixed differential operators coincide with the
$\nabla$-differential operators. We also obtain various properties of mixed differential operators based on (FFC).
Among of them, we obtain a finite generation property for the algebra of bounded mixed differential operators.
We also describe equivalent definitions of $\nabla$-Sobolev spaces in terms of $\nabla$-differential and mixed
differential operators. In the last part of this section, we introduce the notion of $\nabla$-bidifferential operator
(a bilinear version of a differential operator),
and provide a parallel discussion to that of $\nabla$-differential operators, thus obtaining
analogies between the properties of $\nabla$-differential operators and those of $\nabla$-bidifferential
operators. Bilinear differential operators are necessary for a global, geometric discussion of
variational problems.



\subsection*{Acknowledgments}
We thank Herbert Amann, Bernd Ammann, Nadine Gro\ss e, Sergiu Moroianu, and Radu Purice
for useful discussions.

\section{Global differential operators and Sobolev spaces}\label{sec.2}
In this section we recall some basic definitions. For simplicity, we
shall stay as much as possible in the smooth category: smooth
manifolds, smooth vector bundles, smooth coefficients, ... Many of our
results, however, extend to operators with suitable lower regularity
of the coefficients. More often than not, these extensions to lower
regularity coefficients are straightforward, but we leave them for
another paper.

\subsection{Notation and convention}
All along the paper, $n$ is the dimension of the underlying space: $M$,
$\RR^n$,$\ldots $, and $\mu$ or $2m$ are the orders of our differential
operators. Also, we shall use the following conventions for the many types
of dual spaces that we use. If $T : V \to W$ is a linear map, then $T' : W' \to V'$ 
is its dual, where $V'$ is the dual space to $V$. If $V$ and $W$ are complex vector 
spaces endowed with inner products, then $T^* : W \to V$ is the adjoint of $T : V \to W$. 
In case of real vector spaces, instead of the adjoint we have the {\em transpose} $T^\top : W \to V$.

\subsection{Vector bundles and connections}
It will be convenient to use the language of manifolds and vector
bundles as in \cite{AronszajnMilgram, Hormander3, Seeley59, Taylor1},
for example, but there are many other possible references. In this
paper, we follow \cite{AGN1} to which we refer for concepts not
defined here, as well as for some other details. Thus, in the
following, $M$ will be a smooth Riemannian manifold without
boundary with metric $g$. We shall often consider manifolds that are subsets of some
Euclidean space, in which case, they will typically be denoted by
$\Omega$, possibly decorated with various subscripts. In this
subsection we present known, basic results needed in what follows. See
\cite{AubinBook, BesseBook, HebeyBook, Jost} for more background on
differential geometry and for the unexplained concepts and results.

The space of smooth sections of a vector bundle $E \to M$ will be
denoted $\CI(M; E)$, whereas the space of those sections that in
addition have {\em compact} support will be denoted $\CIc(M; E)$.  All
the vector bundles considered in this paper will be smooth. As usual,
$TM \to M$ is the tangent bundle to $M$ and $T^*M \to M$ is the
cotangent bundle to $M$ (the dual of $TM$). Recall that a connection
$\nabla = \nabla^E$ on $E$ is a first order (linear) differential
operator
  \begin{equation*}
    \nabla^E : \CIc(M; E) \to \CIc(M; T^*M \otimes E)\,, \quad
    \nabla^E (fu) = df \otimes u + f \nabla^E (u)\, ,
  \end{equation*}
where $f \in \CI(M)$ and $u \in \CIc(M; E).$ (All differential
operators in this paper will be linear.)  If $X \in \CI(M; TM)$ is a
smooth vector field, then $i_X : \CIc(M; T^*M \otimes E) \to \CIc(M;
T^*M \otimes E)$ is the contraction with $X$ and we have $\nabla^E_X
:= i_X \circ \nabla^E$. If $E, F \to M$ are two vector bundles endowed
with connections $\nabla^E$ and $\nabla^F$ respectively, then we shall
endow the tensor product vector bundle $E\otimes F$ with the induced
connection: $\nabla_X^{E \otimes F} = \nabla_X^{E} \otimes 1 + 1
\otimes \nabla_X^{F}.$ Let $\tau : E \otimes T^*M \otimes F \to T^*M
\otimes E \otimes F$ be the natural isomorphism permuting the first
two factors. Then $\nabla^{E \otimes F} = \nabla^E \otimes 1 + \tau
\circ (1 \otimes \nabla^F)$, which we shall write, by abuse of
notation, in the form
\begin{equation*}
  \nabla^{E \otimes F} = \nabla^{E} \otimes 1 + 1 \otimes
  \nabla^{F}\,.
\end{equation*}

We shall proceed similarly with endomorphism bundles.

\begin{remark}\label{rem.better}
Let $E$ and $F$ be two complex vector bundles on $M$ endowed with
connections.  We endow $\Hom(E, F)$ with the induced connection. More
precisely, let $\kappa : T^*M \otimes \Hom(E; F) \simeq \Hom(E; T^*M
\otimes F)$ denote the natural isomorphism. Then the connection on
$\Hom(E; F)$ is such that, for all $u \in \CIc(M; E)$ and $a \in
\CIc(M; \Hom(E; F))$,
\begin{equation*}
  \nabla^{F} (a u) = \kappa \nabla^{\Hom(E; F)}(a) u + (1 \otimes a)
  \nabla^{E} u\,,
\end{equation*}
The natural morphism $\kappa$ will be omitted from the notation from
now on. In particular, this construction for $F = \CC$, yields the
connection on the dual bundle $E':= \Hom(E; \CC)$, where the trivial
bundle is endowed with the trivial connection.  Thus, if we denote by
$\langle \, , \, \rangle : E' \otimes E \to \CC$ the natural pairing,
then the connection $\nabla^{E'}$ is such that, for all vector fields
$X$ and all smooth, compacted supported sections $u$ and $w$ of the
vector bundles $E$ and $E'$, we have
\begin{equation*}
  \langle \nabla_X^{E'} u, w \rangle \seq X \langle u, w
  \rangle q - \langle u, \nabla_X^{E} w \rangle\,.
\end{equation*}
\end{remark}

Whenever there is no danger of confusion, we shall drop the
superscripts of the connection, thus write $\nabla = \nabla^E$.  We
shall use the notation
\begin{equation*}
    V^{\otimes k} \ede \underbrace{V \otimes V \otimes \ldots \otimes
      V}_{\rm{k-times}} \ \mbox{ and } \ V^{\otimes 0} \ede \CC\,.
\end{equation*}
In particular, $T^{*\otimes j}M := (T^*M)^{\otimes j}$ will denote the
repeated tensor products of the cotangent space $T^*M$, which appears
in the range of the iterated connection map
\begin{equation*}
    \nabla^j \ede \nabla^{T^{* \otimes (j-1)} M \otimes E} \circ
    \ldots \circ \nabla^{T^{*} M \otimes E} \circ \nabla^{E} : \CIc(M;
    E) \to \CIc(M; T^{*\otimes j}M \otimes E)\,,
\end{equation*}
where $T^*M$ and $TM$ are endowed with the Levi-Civita connection
$\nabla = \nabla^{TM} = \nabla^{LC}$. The Levi-Civita connection is
the unique torsion-free, metric preserving connection on $TM$, a
concept that we recall next.

A {\em hermitian} vector bundle $E \to M$ is a complex vector bundle
endowed with a (smoothly varying, sesquilinear) inner product $(\cdot,
\cdot)_E$. Its bounded sections are denoted $L^\infty(M; E)$. A
connection $\nabla = \nabla^E : \CIc(M; E) \to \CIc(M; T^*M \otimes
E)$ is called {\em metric preserving} if, for all $\xi, \eta \in
\CIc(M; E)$, we have
\begin{equation}
\label{covariant-derivative}
   X(\xi, \eta)_E \seq (\nabla_X^E \xi, \eta)_E + (\xi, \nabla_X^E
   \eta)_E \,.
\end{equation}
The space of bounded sections $u$ of $E$ such that all their covariant
derivatives $\nabla^j u \in \CI(M; T^{*\otimes j} M \otimes E)$ are
also bounded is denoted by $\CIb(M; E)$. If the curvature of $E$ and
all its covariant derivatives are bounded (i.e. if the curvature is in
$\CI$), we shall say that $E$ has {\em totally bounded curvature}
\cite{AGN1}.

\subsection{Global, geometric $\nabla$-Sobolev spaces}
By $\dvol$ we shall denote the induced volume form (that is, measure)
on $M$ associated to the metric $g$ on $M$.
We then let, as usual,
\begin{equation}\label{eq.def.normp}
  \|u\|_{L^p(M; E)} \ede
  \begin{cases}
  \ \Big ( \int_{M} \|u(x)\|_E^p \dvol(x) \Big)^{1/p} & \mbox{ if } 1
  \le p < + \infty\\
  \ \inf_{\dvol(N) = 0} \sup_{x \in M \smallsetminus N} \|u(x)\|_E &
  \mbox{ if } p = + \infty
  \end{cases}
\end{equation}
(of course $\|u\|_{L^\infty(M; E)}$ is the essential supremum,
$\esssup \|u(x)\|_E$, of $u$). As usual, we identify sections of $E$
that coincide outside a set of measure zero to define the
$L^p$--spaces:
\begin{equation}\label{eq.def.Lp}
  L^p(M; E) \ede \{ u : M \to E \where \|u\|_p < + \infty \}/\ker (\|
  \cdot \|_p)\,.
\end{equation}

\subsubsection{Definition of $\nabla$-Sobolev spaces}
We now introduce the Sobolev spaces in which we are interested in a
global way, as in \cite{AGN1, HebeyBook}. We need to use the index
$\nabla$ in their notation since there will be other definitions of
Sobolev spaces as well and the definition depends on the choice of the
connection (in general). We consider only complex Sobolev
  spaces, but the real case can be treated similarly.

\begin{definition}\label{def.Sobolev}
Let $M$ be a Riemannian manifold with metric $g$ and volume form
$\dvol$. Let $E \to \CC$ be a finite dimensional, hermitian vector
bundle with metric preserving connection $\nabla^E$. We extend
$\nabla^E$ to connections $\nabla$ on the bundles $T^{*\otimes k}M
\otimes E$, $k \in \NN$, using the Levi-Civita connection
$\nabla^{LC}$ on $TM$ (we omit the superscrit to lighten the
notation). Let $s \in \NN$. Then
\begin{equation*}
  W^{s, p}_{\nabla }(M; E) \ede \{u \where \nabla^j(u) \in
  L^p(M;T^{*\otimes j } M \otimes E) \,,\, \mbox{ for }\, 0 \le j \le
  s \}
\end{equation*}
is the {\em order $s$, $L^p$--type $\nabla$--Sobolev space} of
sections of $E$ (so $W_\nabla ^{0, p} = L^p$) with norm
\begin{equation}
\label{Sobolev-norm}
\|u\|_{W^{s, p}_{\nabla}(M; E)} \ede \ell^p\mbox{--norm of } \{
\|\nabla^j (u)\|_{L^p(M;T^{* \otimes j} M \otimes E)}\,,\ 0 \le j
\le s \} \,.
\end{equation}
\end{definition}

When there is no danger of confusion, we shall write $\|u\|_{L^p(M)}$
{and even $\|u\|_{L^p}$} for $\|u\|_{L^p(M; E)}$ and
$\|u\|_{W^{s,p}_{\nabla}(M)}$ {(or even
  $\|u\|_{W^{s,p}_{\nabla}}$)} for $\|u\|_{W^{s, p}_{\nabla}(M; E)}$.
In particular, $\|u\|_{W^{s, \infty}_{\nabla}(M)} \ede \max_{j=0}^s
\|\nabla^j (u)\|_{L^\infty(M)}$.  We let $H_{\nabla}^s(M; E) :=
W_{\nabla}^{s, 2}(M; E)$, thus $\|u\|_{H^s_{\nabla}(M;E)}^2 \ede
\sum_{j=0}^s \|\nabla^j (u)\|_{L^2(M)}^2$.

As we shall see shortly below, the spaces $W^{s, p}_\nabla$ do have some of 
the usual properties of the Sobolev spaces on compact, smooth manifolds (with 
or without boundary) provided that $M$ has bounded geometry, see for instance 
\cite{AGN1, GrosseSchneider, HebeyBook, TriebelBG} and the references therein. 
See \cite{GM1, GM2} for related papers using Sobolev spaces that go beyond the bounded geometry setting.
The reader should be cautioned, however, not to take everything for granted since, 
for instance, the spaces $W^{s,p}_\nabla$ do depend on the choice of $\nabla$. 
See, for instance, Example \ref{ex.one} on ``magnetic Sobolev spaces.''

The space
\begin{equation}
\CIb(M; E) \ede W_{\nabla}^{\infty, \infty}(M; E) \ede \bigcap_{s
\in \NN} W_{\nabla}^{s,\infty }(M; E) \,,
\end{equation}
introduced above, that is, the space of bounded sections of $E$ such
that all their covariant derivatives are also bounded, will play an
important role in what follows. It is a Fr\'echet space endowed with
the increasing family of semi-norms $\| \cdot
\|_{W_{\nabla}^{s,\infty}(M; E)}$. Recall that a subset $S \subset
\CIb$ is bounded if, and only if, it is bounded in every semi-norm
(see, e.g., \cite[Theorem 1.37]{Rudin}).

\subsubsection{Negative and non-integer order $\nabla$-Sobolev spaces}
Let next $W_{0,\nabla }^{s,p}(M; E)$, $s \in \NN$, be the closure of
the space $\CIc(M;E)$ in $W_\nabla ^{s,p}(M; E)$. If $M$ is complete,
then, for any $1 \le p < \infty$, $W_{0,\nabla }^{1,p}(M;E)=W_\nabla
^{1,p}(M; E)$ (see \cite[Theorem 2.7]{HebeyBook}). Moreover, if $s\in
\NN$, $s\geq 2$, and $M$ is complete with positive injectivity radius
and Ricci curvature bounded up to the order $s-2$, then, for any $1
\le p < \infty$, $W_{0,\nabla }^{s,p}(M;E) = W_\nabla ^{s,p}(M; E)$
(see \cite[Theorem 2.8]{HebeyBook}).

We shall use the spaces $W_{0,\nabla }^{s,p}(M; E)$, for
$1<p<+\infty$, to define the Sobolev spaces with negative index
\begin{equation}\label{eq.def.neg.Sob}
W^{-s,p}(M; E^*) \ede W_{0,\nabla }^{s,p'}(M; E)^* \,,
\end{equation}
where $V^*$ is the complex conjugate dual of $V$ and
$\frac{1}{p}+\frac{1}{p'}=1$. For simplicity, in the following, we
shall identify $E^*$ with $E$ using the hermitian metric on $E$, so
$W^{-s,p}(M; E^*)\simeq W^{-s,p}(M; E)$. We also define the spaces
$W_{\nabla }^{s,p}(M; E)$ for $s \notin \ZZ$ by complex interpolation
(see \cite{LionsMagenes1, TriebelBook} for further details).

\subsubsection{Weighted $\nabla$-Sobolev spaces}
One of the main reason for our interest in general Sobolev spaces on manifolds is
that they are useful in the study of {\em weighted} Sobolev spaces, whose definition
we recall next (compare with Definition \ref{def.Sobolev}).

\begin{definition}
\label{rem.weighted}
Let $\rho, f_0: M \to (0, \infty)$ be given. Let $s\in {\mathbb N}$ and $p\in [1,+\infty ]$. Then
\begin{equation}
\label{weight-Sobolev}
f_0 W^{s, p}_{\nabla, \rho} (M; E) \ede \{u \where \rho^j \nabla^j
(f_0^{-1} u) \in L^p(M; E) \,,\, \mbox{ for }\, 0 \le j \le s\}\,,
\end{equation}
is the {\em order $s$}, $L^p$--type {\em $\nabla$--weighted Sobolev space} of sections of $E$,
 (in particular, $f_0W_{\nabla, \rho}^{0, p} = f_0L^p$), endowed with the norm
\begin{equation*}
\|u\|_{f_0 W^{s, p}_{\nabla, \rho} (M)} \ede \ell^p\mbox{--norm of }
\{ \|\rho^j \nabla^j (f_0^{-1} u)\|_{L^p(M)}\,,\ 0 \le j \le s \}
\,.
\end{equation*}
\end{definition}

We let $f_0H_{\nabla, \rho}^s(M; E):=
f_0W_{\nabla, \rho}^{s, 2}(M; E)$, and, thus, for any $u\in f_0H_{\nabla, \rho}^s(M; E)$,
$\|u\|_{f_0H^s_{\nabla, \rho}(M;E)}^2 \ede
\sum_{j=0}^s \|\rho^j \nabla^j (f_0^{-1} u)\|_{L^p(M)}^2$.

These weighted spaces are quite important in applications to geometry
\cite{AldanaCarron, AGN3, GM2, Gounoue, GN17, MWeyl} or to
PDEs \cite{BNZ3D1, Mironescu, CDN12, 
daugeBook, RelFluids, Kondratiev67, KMR}. Their study is similar to that of the standard 
(unweighted) Sobolev spaces. One is lead
to consider them even if one is interested only on PDEs on domains in $\RR^n$ and they can be studied by relating
them to Sobolev spaces on manifolds. If $f_0 = \rho = 1$, these spaces, of course, reduce to the standard
$L^p$--type Sobolev space $W_{\nabla}^{s, p} = 1W^{s,p}_{\nabla, 1}$ considered above. Often, the study of
weighted Sobolev spaces can be reduced to that of usual Sobolev spaces using conformal changes of metric, see,
for instance, \cite{AmannFunctSp,AGN_CR, AGN.wip}.

To illustrate this property, we denote by $g$ the metric tensor of $M$ with the associated Levi-Civita connection
$\nabla^{LC}$, and assume that $\rho ,f_0:M\to (0,\infty )$ are measurable functions. Let $g_0:=\rho ^{-2}g$. Assume
that $\rho $ is an {\it admissible weight} with respect to the metric $g_0$ (that is, $\rho $ is smooth and
$\rho ^{-1}d\rho \in \CIb(M; T^*M)$) and that $f_0$ is continuous.
Then for $\ell \in {\mathbb N}$ and $p\in [1,\infty ]$, we consider the weighted Sobolev space 
$f_0 W^{\ell , p}_{\nabla, \rho} (M;E)$ and the classical Sobolev space $W^{\ell , p}(M,g_0;E)$ 
(defined with respect to the new metric $g_0$). The relation between these spaces is given by the formula
\begin{equation}
\label{weight-classic}
    f_0 W^{\ell , p}_{\nabla, \rho} (M;E) = f_0\rho ^{-\frac{n}{p}}W^{\ell , p}(M,g_0; E),\ \
    \forall \ \ell \in {\mathbb N},\ p\in [1, \infty]
\end{equation}
\cite{AmannFunctSp, sobolev}, and, especially, Equation (1) in \cite{AGN_CR} (recall that we are assuming $\rho$ 
to be admissible). This relation between weighted and unweighted Sobolev spaces provides a strong motivation for 
our study, even if one is not interested in PDEs on manifolds, since weighted spaces appear naturally in the study of
polyhedral domains (thus even in the flat space).

\subsubsection{Negative order $\nabla$-weighted Sobolev spaces}
Let $\rho, f_0: M \to (0, \infty)$ be given. Let $s \in \NN$, $1<p<+\infty$, and
$f_0W_{0,\nabla ,\rho }^{s,p}(M; E)$ be the closure of the space $\CIc(M;E)$ in $f_0 W_{\nabla, \rho}^{s,p}(M; E)$.
We use the spaces $f_0W_{0,\nabla ,\rho }^{s,p}(M; E)$ to define the weighted Sobolev space  with negative index
\begin{equation}
\label{eq.def.neg.weight.Sob}
    f_0 W^{-s,p}_{\nabla ,\rho }(M; E^*) \ede \left(f_0^{-1} W_{0,\nabla ,\rho }^{s,p'}(M; E)\right )^* \,,
\end{equation}
where $\frac{1}{p}+\frac{1}{p'}=1$. Recall that we identify $E^*$ with $E$ using the hermitian metric on $E$, so
$f_0 W_{\nabla ,\rho }^{-s,p}(M; E^*)\simeq f_0 W_{\nabla ,\rho }^{-s,p}(M; E)$.


\subsection{First properties of $\nabla$-Sobolev spaces}
We now prove a multiplicative property of our Sobolev spaces

\begin{proposition}
\label{prop-Sobolev-new}
Let $\ell \in \NN$ and all $p,q,r \in [1, \infty]$ be such that $\frac{1}{p}+\frac{1}{q}=\frac{1}{r}$.
Then there is a constant
$C_{\ell, p,q} >0$ such that, for all $M$, all $E, E_1 \to M$, all
connections $\nabla^E$ and $\nabla^{E_1}$, and all $(a, u)\in
W_{\nabla}^{\ell,p}(M; \Hom(E; E_1)) \times
W_{\nabla}^{\ell,q}(M; E)$, we have
\begin{equation*}
  \|au\|_{W_{\nabla}^{\ell,r}(M; E_1)} \le C_{\ell, p,q}
\|a\|_{W_{\nabla}^{\ell,p}(M; \Hom(E; E_1))}
\|u\|_{W_{\nabla}^{\ell,q}(M; E)}\,.
\end{equation*}
In particular, the evaluation in $E$ defines a continuous bilinear map
\begin{equation*}
    W_{\nabla}^{\ell,p}(M; \Hom(E; E_1))\times
    W_{\nabla}^{\ell,q}(M; E) \ni (a, u)\, \mapsto \, au \in
    W_{\nabla}^{\ell,r}(M; E_1)\,.
\end{equation*}
\end{proposition}

\begin{proof}
Let us assume that at least one of the indices $p$ and $q$, for instance, $p$ belongs to $[1,\infty )$.
The case $p=\infty $ is similar, but easier. We show our property by induction on $\ell $. If $\ell = 0$
the result is known. Indeed, if $(a,u)\in L^p
(M; \Hom(E; E_1))\times L^q(M; E)$, then $au \in L^{r}(M;E_1)$, by the multiplication property
$L^{p}(M; \Hom(E; E_1))\times L^{q}(M; E)\hookrightarrow L^{r}(M; E_1)$. Moreover, we have
\begin{equation*}
    \|au\|_{L^{r}(M)}\leq \|a\|_{L^p(M)}\|u\|_{L^q(M)}\,,
\end{equation*}
and hence the bilinear map is continuous and $C_{0,p,q} = 1$ is
independent of all choices.

Assume now that our statement is valid for $\ell-1 \geq 0$, and show
that it holds also for $\ell$.  Let $(a, u)\in
W_{\nabla}^{\ell,p}(M; \Hom(E; E_1)) \times
W_{\nabla}^{\ell,q}(M; E)$. Then in view of the embedding
$W_{\nabla}^{\ell,p}(M; \Hom(E; E_1)) \times
W_{\nabla}^{\ell,q}(M; E)\hookrightarrow W_{\nabla}^{\ell-1,p}(M;
\Hom(E; E_1))\times W_{\nabla}^{\ell-1,q}(M; E)$ and by the induction
hypothesis we obtain that $au\in W_{\nabla}^{\ell-1,r}(M; E_1)$. Thus,
\begin{align}
\label{au-1}
   \nabla ^{j}(au)\in L^{r}(M; T^{* \otimes j }M \otimes
   E_1),\ \ 0\leq j\leq \ell -1\,.
\end{align}
It remains to show that $\nabla ^{\ell }(au)\in L^{r}(M; T^{* \otimes
  \ell}M \otimes E_1)$. To this end, we use the formula
\begin{align}
\label{au}
   \nabla(au) = \nabla(a)u + (1\otimes a)\nabla u
\end{align}
(see Remark \ref{rem.better}) and obtain that
\begin{equation}
\label{au-2}
\nabla ^{\ell }(au) =\nabla ^{\ell -1}\left(\nabla (au)\right)
=\nabla ^{\ell -1}\left(\nabla(a)u + (1\otimes a)\nabla u\right)\,.
\end{equation}
We have $\nabla a\in W_{\nabla}^{\ell -1,p}(M; \Hom(E, T^*M
\otimes E_1))$ and $u\in W_{\nabla}^{\ell -1,q}(M; E)$ and, thus,
$\nabla(a)u\in W_{\nabla}^{\ell -1,r}(M; E_1)$ by the induction
hypothesis. The same argument gives that $(1\otimes a)\nabla u\in
W_{\nabla}^{\ell -1,r}(M; E_1)$. Therefore, $\nabla ^{\ell }(au)\in
L^r(M; E_1)$ by formulas \eqref{au-1} and \eqref{au-2} with continuous
dependence. Hence $au\in W_{\nabla}^{\ell ,r}(M; E_1)$ and the induced
map is continuous, as asserted. More precisely, for all $p < \infty$, $q,r\in [1,\infty ]$, 
such that $\frac{1}{p}+\frac{1}{q}=\frac{1}{r}$, we obtain that
\begin{align*}
\|au\|_{W_\nabla ^{\ell, r}(M)}^r & \le \, \|au\|_{W_\nabla
^{\ell-1, r}(M)}^r + \|\nabla(au)\|_{W_\nabla ^{\ell-1,r}(M)}^r\\
& \leq C_{\ell-1, p,q} \|a\|_{{W_\nabla ^{\ell-1,p}(M)}}^r\|u\|_{{W_\nabla ^{\ell-1, q}(M)}}^r
+ \|\nabla(a)u + (1\otimes a)\nabla u\|_{W_\nabla ^{\ell-1, r}(M)}^r\\
%
%
& \leq C_{\ell-1, p,q} \Big(\|a\|_{W_\nabla ^{\ell-1,p}(M)}^r 
\|u\|_{W_\nabla ^{\ell-1, q}(M)}^p + 2^{r-1}\big(\|\nabla a\|_{W_\nabla ^{\ell-1,p}(M)}^r\|u\|_{W_\nabla ^{\ell-1,
q}(M)}^r\\
& \hspace{3em} + \|1\otimes a\|_{W_\nabla ^{\ell-1,p}(M)}^r\|\nabla u \|_{W^{\ell-1, q}(M)}^r\big)\Big)\\
%
%
%
%
&\leq C_{\ell-1, p,q} \left(1+2^{r}\right)\|a\|_{W_\nabla ^{\ell ,p}(M)}^r\|u\|_{W_\nabla ^{\ell , q}(M)}^r\\
& = C_{\ell ,p,q}^r \|a\|_{W_\nabla ^{\ell ,p}(M)}^r\|u\|_{W_\nabla ^{\ell , q}(M)}^p\,,
\end{align*}
where $C_{\ell ,p,q}^r:= C_{\ell-1, p,q}(1+2^{r})$.
\end{proof}

\begin{corollary} \label{prop-Sobolev}
Let $\ell \in \NN$ and all $q \in [1, \infty]$. Then there is a constant
$C_{\ell, q} >0$ such that, for all $M$, all $E, E_1 \to M$, all
connections $\nabla^E$ and $\nabla^{E_1}$, and all $(a, u)\in
W_{\nabla}^{\ell,\infty}(M; \Hom(E; E_1)) \times
W_{\nabla}^{\ell,q}(M; E)$, we have
\begin{equation*}
\|au\|_{W_{\nabla}^{\ell,q}(M; E_1)} \le C_{\ell, q}
\|a\|_{W_{\nabla}^{\ell,\infty}(M; \Hom(E; E_1))}
\|u\|_{W_{\nabla}^{\ell,q}(M; E)}\,.
\end{equation*}
In particular, the evaluation in $E$ defines a continuous bilinear map
\begin{equation*}
    W_{\nabla}^{\ell,\infty}(M; \Hom(E; E_1)) \times
    W_{\nabla}^{\ell,q}(M; E) \ni (a, u)\, \mapsto \, au \in
    W_{\nabla}^{\ell,q}(M; E_1).
\end{equation*}
\end{corollary}

Recall that $g$ is a metric of the $n$-dimensional manifold $M$ and $\nabla^{LC}$ is
the associated Levi-Civita connection. Assume that $\rho ,f_0,h_0:M\to (0,\infty )$ are
measurable functions. Let $g_0:=\rho ^{-2}g$ and $\nabla _0$ be the associated Levi-Civita
connection. Assume that $\rho $ is an admissible weight with respect to the metric $g_0$.
Let ${\rm{grad}}_{g_0}\phi $ denote the vector field which represents the image of $d\phi $ in $TM$
under the metric $g_0$.
If $\rho =e^\phi $ then we have the following the relation between $\nabla^{LC}$ and $\nabla _0$:
\begin{align}
\label{LC-0}
    (\nabla^{LC}-\nabla _0)_X Y \seq X(\phi )Y + Y(\phi )X -g_0(X,Y)
    \operatorname{grad}_{g_0}\phi\,,
\end{align}
(see, for instance, formula (5) in \cite{AGN2020}).

Let $W_0^{\ell ,p}(M,g_0;E)$ denote the classical Sobolev space defined as the completion of the space 
$\CIc(M; E)$ with respect to the norm given by \eqref{Sobolev-norm} and corresponding to the metric $g_0$.
Then formula \eqref{weight-classic} and Proposition \ref{prop-Sobolev-new} applied to the classical Sobolev
spaces $W^{\ell ,p}(M,g_0;E)$ lead to the following multiplication property of weighted Sobolev spaces.

\begin{proposition}
\label{prop-Sobolev-weight}
Let $\ell \in \NN$ and $p,q,r \in [1, \infty]$ be such that
$\frac{1}{p}+\frac{1}{q}=\frac{1}{r}$. Let $f_0,h_0:M\to (0,\infty )$ be given measurable functions.
If $\rho : M \to (0, \infty)$ is an admissible weight with respect to the metric $g_0=\rho ^{-2}g$ and
$f_0,h_0$ are continuous, then the bilinear map
\begin{equation*}
   f_0 W^{\ell , p}_{\nabla, \rho}(M; \Hom(E; E_1))\times 
   h_0W_{\nabla, \rho}^{\ell,q}(M;E) \ni (a, u)\, \mapsto \, au \in
   f_0 h_0 W_{\nabla, \rho}^{\ell,r}(M; E_1)
\end{equation*}
is continuous.
\end{proposition}

\begin{proof}
In view of Proposition \ref{prop-Sobolev-new} we obtain the continuity of the embedding
(of the product of classical Sobolev spaces defined with respect to the metric $g_0$)
$$W^{\ell , p}(M,g_0; \Hom(E; E_1))\times W^{\ell,q}(M,g_0;E)\hookrightarrow W^{\ell,r}(M,g_0;E)\,.$$
Then formula \ref{weight-classic} implies the continuity of the embedding
\begin{align*}
&f_0 W^{\ell , p}_{\nabla^{LC}, \rho}(M; \Hom(E; E_1))\times
h_0W_{\nabla^{LC}, \rho}^{\ell,q}(M;E)\\
&\hspace{2em}=f_0\rho ^{-\frac{n}{p}}W^{\ell , p}(M,g_0;\Hom(E; E_1))
\times h_0\rho ^{-\frac{n}{q}}W^{\ell,q}(M,g_0;E)\\
&\hspace{2em}\hookrightarrow (f_0h_0)\rho ^{-(\frac{n}{p}+\frac{n}{q})}W^{\ell,r}(M,g_0;E)\\
&\hspace{2em}=(f_0h_0)\rho ^{-\frac{n}{r}}W^{\ell,r}(M,g_0;E)\\
&\hspace{2em}=f_0h_0W_{\nabla^{LC}, \rho}^{\ell,r}(M;E)\,,
\end{align*}
that is, the desired result.
\end{proof}

To understand the complications introduced by the use of the connection, let us discuss the case
of ``magnetic Sobolev spaces'' \cite{MSSp1, Purice2007}.  Let us introduce first some notation.

\begin{notation}\label{not.local2}
\normalfont{Assume $M \subset \RR^n$, so we have global coordinates
  $x_j$. We shall use the following notation}
\begin{enumerate}[(i)]
\item $I \define \{1, 2, \ldots, n\}$ and $J_{\mu} \define
  \{\emptyset\} \cup I \cup I^2 \cup \ldots \cup I^{\mu}$;

\item $(e_j := \pa_j := \frac{\pa }{\pa x_j})_{j \in I}$ and $(e_j^*
  := dx_j)_{j \in I}$, are the standard basis of $T \Omega$ and,
  respectively $T^*\Omega$, and $e_{\vect{i}} := e_{i_0} \otimes
  e_{i_1} \otimes \ldots e_{i_r}$ and $e^*_{\vect{i}} := e^*_{i_0}
  \otimes e^*_{i_1} \otimes \ldots e^*_{i_r}$, where $\vect{i} =(i_1,
  i_2, \ldots, i_r) \in I^r \subset J_\mu $ and $e_\emptyset =
  e^*_\emptyset = 1 \in \CC$;

\item For $\vect{i} = (i_1, i_2, \ldots,i_r) \in I^r \subset J_\mu$,
  let $|\vect{i}| := r$

\item\label{not.item.iii}
For $\vect{i} = (i_1, i_2, \ldots,i_r) \in I^r$, let
$\nabla_{\vect{i}} \ede \nabla_{i_1} \nabla_{i_2} \ldots
\nabla_{i_r}$.
\end{enumerate}
\end{notation}

\begin{example}[Magnetic Sobolev spaces]\label{ex.one}
Let us assume that $M = \RR^n$ with the flat metric, but that the trivial bundle $E = M
\times V \to M$ has a non-trivial connection, where $V$ is some
given real vector space. Let $\nabla_j = \nabla_{e_j}$ be the covariant derivative
with respect to the vector field $e_j := \frac{\pa}{\pa x_j}$. Then
$\nabla_j (dx_k) = 0$ and
\begin{equation}\label{eq.ex.one.1}
    \nabla_j(\xi) \ede \pa_j \xi + A_j \xi \, ,
\end{equation}
where $\pa_j$ is the partial derivative with respect to $j$th variable
(which is defined since $E \to M = \RR^n$ is a trivial bundle)
and $A_j \in \CI(M; \End(V))$. Recall the notation of
\ref{not.local2}. We agree that $\nabla_{\emptyset} u = u$.

By induction, we then obtain that
    \begin{equation}\label{eq.ex.one.2}
        \nabla^r u \seq \sum_{\vect{i} \in J^r} e^*_{\vect{i}} \otimes
        \nabla_{\vect{i}} u\, .
  \end{equation}
We obtain, in particular, that
\begin{equation}\label{eq.actully.important}
    u \in W^{r,p}_{\nabla}(M; E) \ \Leftrightarrow \ \forall \vect{i}
    \in J_r\,,\ \nabla_{\vect{i}} u \in L^p(M; T^{*\otimes |\vect{i}|}
    M \otimes E)\,.
\end{equation}
Because of the lack of commutation (in general) of the operators
$\nabla_i$, these conditions may become more stringent than in the
classical case.  Let us assume furthermore that $n = 2$, $V = \CC^2$,
$A_1 = 0$, and
\begin{equation}\label{eq.ex.one.3}
    A_2 \ede \left (
    \begin{array}{cc}
    0 & e^{\imath x_1^3}\\
    - e^{-\imath x_1^3} & 0
  \end{array}
  \right) \,.
\end{equation}
(Notice that $A_2^* = - A_2$.) Then $H^1_{\nabla}(M; E) = H^1(M; E)$,
but $H^2_{\nabla}(M; E) \neq H^2(M; E)$. Indeed, let us use
Equation \eqref{eq.ex.one.1}. We shall write $\nabla_{\vect{i}} = \nabla_{(1,1)} = \nabla_{e_1}^2$
when $\vect{i} = (1, 1)$, and so on, according to \ref{not.local2}. Then we obtain
\begin{align*}
    &\nabla _{(1,1)}\xi =\partial _1^2\xi \,,\
\nabla _{(1,2)}\xi =\partial _1\partial _2\xi + \left(
3\imath x_1^2e^{\imath x_1^3}\xi _2 + e^{\imath x_1^3}\partial _1\xi _2,\
3\imath x_1^2e^{-\imath x_1^3}\xi _1-e^{-\imath x_1^3}\partial _1\xi _1
\right)\\
&\nabla _{(2,1)}\xi =\partial _2\partial _1\xi  + \left(
e^{\imath x_1^3}\partial _1\xi _2,
-e^{-\imath x_1^3}\partial _1\xi _1
\right),\,
\nabla _{(2,2)}\xi =\partial _2^2\xi +
2\left(
e^{\imath x_1^3}\partial_2\xi _2,
-e^{-\imath x_1^3}\partial _2\xi _1
\right)-\xi .
\end{align*}
See \cite{Purice2007, MSSp1}.
\end{example}

We continue with some further properties of $\nabla$-Sobolev spaces.

\begin{corollary}
The tensor product defines a continuous bilinear map
\begin{equation}
\label{eq.mult.Sobolev}
   W_\nabla ^{\ell,\infty}(M; E) \times W_\nabla ^{\ell,p}(M; E_1) \ni
   (u,v) \, \mapsto \, u \otimes v \in W_\nabla ^{\ell,p}(M; E \otimes
   E_1) \,.
\end{equation}
\end{corollary}

\begin{proof}
The natural map $E \to \Hom(E_1, E \otimes E_1)$ is contractive, so the
continuity of the map \eqref{eq.mult.Sobolev} follows from Proposition
\ref{prop-Sobolev}.
\end{proof}

\begin{proposition} Let $E, F \to M$ be hermitian vector bundles
  endowed with metric preserving connections. Then the pointwise trace
  $\tr_x\colon E_x \otimes E_x' \to \CC$, $x\in M$, $\tr_x(v' \otimes
  v) := \langle v', v \rangle$, defines the canonical contraction
  $\epsilon : E \otimes E' \otimes F\to F$, which satisfies $\nabla
  (\epsilon) = 0$. Consequently, $\epsilon \in W^{\infty, \infty}(M;
  \Hom( E \otimes E' \otimes F; F))$ and it thus induces a continuous
  map
\begin{equation}\label{eq.cont.Sobolev}
     \epsilon : W_{\nabla}^{\ell, p}(M; E \otimes E' \otimes F) \ \to
     \ W_{\nabla}^{\ell, p}(M; F)\,.
\end{equation}
\end{proposition}

\begin{proof} We have $\|\tr_x\| \le 1$ by the definition of
  the norm on $E'$ (we endow the tensor product with the projective
  tensor product norm). Consequently $\|\epsilon\|_\infty \le 1$.  The
  relation $\nabla (\epsilon) = 0$ follows from the definition of the
  connection on $E'$ and on tensor products. Indeed, using Remark
  \ref{rem.better}, we obtain
  \begin{align*}
    \nabla(\epsilon)(v' \otimes v \otimes \xi) & = \nabla \big[
      \epsilon(v'\otimes v \otimes \xi)\big] - \epsilon\big[ \nabla
      (v'\otimes v \otimes \xi)\big]\\
& = \nabla \big (\langle v', v\rangle \otimes \xi \big) -
    \epsilon\big[ \nabla (v' \otimes v) \otimes \xi + v' \otimes v
      \otimes \nabla(\xi) \big]\\
 & = \big [ d\langle v', v \rangle + \epsilon \nabla (v'\otimes v)
      \big ] \otimes \xi + \langle v', v \rangle \otimes \nabla(\xi)
    - \epsilon\big( v' \otimes v \otimes \nabla(\xi) \big) \\
 & = \big [ d\langle v', v \rangle + \langle \nabla(v'), v \rangle +
      \langle v', \nabla(v) \rangle \big ] \otimes \xi \seq 0 \,.
    \end{align*}
  Therefore $\nabla^{\ell}(\epsilon) = 0 \in L^\infty$, for all
  $\ell$.

  Once we have established that $\epsilon \in W^{\infty, \infty}(M;
  \Hom( E \otimes E' \otimes F; F))$, the second part of the proposition
  follows from Proposition \ref{prop-Sobolev}.
\end{proof}

Although we shall not use the following result in this paper, it is an
interesting general result that may allow us to compare Sobolev norms
with their Euclidean counter-parts.

Let $\maU = (U_i)_{i \in I}$ be a covering of $M$, that is, $M =
\cup_{i \in I} U_i$. Recall that its {\em covering multiplicity} is
$N(\maU) = N((U_i)_{i \in I})$
\begin{equation}
  \label{eq.def.c.i}
     N(\maU) \ede \max \{r \, \vert \ \exists i_1 , i_2, \ldots, i_r
     \in I \mbox{ distinct with } U_{i_1} \cap U_{i_2} \cap
     \ldots \cap U_{i_r} \neq \emptyset \}\,.
\end{equation}
Also, recall that $\maU = (U_i)_{i \in I}$ of $M$ is {\em uniformly
  locally finite} if $N(\maU) < \infty$. Assume that the index set $I$
labeling the open sets of the covering $\maU$ is countable. Then we
also let
\begin{equation}\label{eq.norm3}
  ||| u |||_{\maU, s, p} \ede
  \begin{cases}
    \, \Big( \sum_{j \le s, i \in I} \|\nabla^i u\|^p_{L^p(U_i)}
    \Big)^{1/p} & \mbox{ if } 1 \le p < \infty\\
    \ \sup_{j \le s, i \in I} \|\nabla^i u\|_{L^\infty(U_i)} &
    \mbox{ if } p = \infty
    \end{cases}
\end{equation}

\begin{proposition}\label{prop.comparison} We
  have $ ||| u |||_{\maU, s, \infty} = \|u\|_{W^{s, \infty}(M)}$. If
  $1 \le p < \infty$, then $\|u\|_{W^{s, \infty}(M)} \le ||| u
  |||_{\maU, s, p} \le N(\maU)^{1/p} \|u\|_{W^{s, \infty}(M)}.$
\end{proposition}

\begin{proof} This follows from the definitions of the
  norms $\| \, \cdot \, \|_{W^{s, \infty}(M)}$ and $||| \, \cdot \,
  |||_{\maU, s, p}$ and, for $f \ge 0$ measurable, the inequalities
\begin{equation*}
    \int_M f \dvol \, \le \, \sum_{i \in I} \, \int_{U_i} f \dvol \, \le
    \, N(\maU) \int_M f \dvol\,.
\end{equation*}
This completes the proof.
\end{proof}

To continue  our study of $\nabla$-Sobolev spaces, we need to take a
look also at ``$\nabla$-differential operators,'' which  we introduce
in the next section.

\subsection{Global, geometric $\nabla$-differential operators}
We now start our study of differential operators. We shall consider
globally defined differential operators on $M$ with smooth
coefficients acting on sections of smooth vector bundles. We provide
several definitions. We are especially interested in definitions that
do not rely on local coordinates (unlike the classical one). We may
assume that these vector bundles are endowed with metrics and
metric preserving connections, denoted generically by $\nabla$, as
agreed above. We let $\nabla^0 = id$.

As mentioned already, {\em all our vector bundles will be smooth.}
Also, {\em we shall consider differential operators with smooth
  coefficients, unless otherwise mentioned.} The case of differential
operators with non-smooth coefficients is to a large extent very
similar.

Recall that $V^{\otimes k} \ede V \otimes V \otimes \ldots \otimes V$
($k$-times) and that $V^{\otimes 0} \ede \CC\,.$ It will be convenient
to consider the ``truncated Fock space''
\begin{equation} \label{eq.def.FmE}
   \maF^M_\mu(E) \define \oplus_{j=0}^\mu\, T^{*\otimes j}M \otimes E
   \,.
\end{equation}
We endow the truncated Fock space $ \maF^M_m(E)$ with the induced
connection from $E$ and $T^*M$. Given $a \in \CI(M; \maF_\mu^M (E);
F))$, we shall write $a^{[j]} \in \CI(M;\Hom(T^{*\otimes j} M \otimes
E; F))$, $j=0,\ldots ,\mu $, for the resulting component and we
shall also define $a \cdot \nabla^{tot}:=\sum_{j=0}^{\mu}a^{[j]}\nabla^j$.

\begin{definition}\label{def.diff.op}
Let $E, F \to M$ be vector bundles, with $E$ endowed with a
connection, and let {$a = (a^{[0]}, a^{[1]}, \ldots, a^{[\mu]})$ be a
  suitable section of $\Hom(\maF_\mu^M (E); F))$.} A {\em
  $\nabla$--differential operator} is a map
\begin{equation*}
   P \seq a \cdot \nabla^{tot} \ede \sum_{j=0}^{\mu}a^{[j]}\nabla^j:
   \CIc(M;E) \to \CIc(M; F)\,.
\end{equation*}
We let $\ord(P)$ denote the least $\mu$ for which such a writing
exists and call it the {\em order} of $P$. Suitable extensions by
continuity of $P$ will also be called $\nabla$--differential operators
and will be denoted by the same letter.
\end{definition}

Unless stated otherwise, all our differential operators will have
smooth coefficients. Thus, by ``a $\nabla$-differential operator''
will mean a ``$\nabla$-differential operator with smooth
coefficients.'' The case of operators with non-smooth coefficients
will only rarely be considered, but it usually can be treated in a
similar way.

\begin{notation}\label{not.Diff}
Let us introduce now some further notation and terminology related a
differential operator $P := a \cdot \nabla^{tot} :=
\sum_{j=0}^{\mu}a^{[j]}\nabla^j$ as in the definition.
\begin{itemize}
  \item If $a \in \CI(M; \Hom(\maF_\mu^M (E); F))$, we shall say
    that $P$ {\em has $\CI$-coefficients.} The set of such operators
    is denoted $\Diff_{\nabla}^\mu(M; E, F)$.

\hspace*{-1.3cm} Assume now also that $E$ and $F$ are endowed with
hermitian metrics.

\item
  If $a \in W_{\nabla}^{\ell,\infty}(M; \Hom(\maF_\mu^M (E); F))$, we
  shall say that $P$ {\em has coefficients in $W^{\ell,\infty}$.}
\item If, in fact, $\ell = 0$, we shall say that $P$ has {\em bounded}
  coefficients.
\item On the other hand, if $\ell = \infty$, we shall say that $P$ has
  {\em totally bounded} (or $\CIb$) coefficients.  The set of such
  operators is denoted {$\Diff_{b, \nabla}^\mu(M; E, F)$.}
\end{itemize}
\end{notation}

We shall often drop the index $\nabla$ from $\Diff_{\nabla}^\mu(M; E,
F)$ and $\Diff_{b, \nabla}^\mu(M; E, F)$.  Here are some comments.

\begin{remark}\label{rem.easy}
We use the notation of Definition \ref{def.diff.op}.
\begin{enumerate}
\item
If $j \ge 2$, the coefficient $a^{[j]}$ in $a = (a^{[0]}, a^{[1]},
\ldots, a^{[\mu]}) \in \CI(M; \maF_\mu^M (E); F))$ is {\em not}
uniquely determined by the map $P = a \cdot \nabla^{tot} :=
\sum_{j=0}^{\mu} a^{[j]} \nabla^j :\CIc(M; E) \to \CIc(M; F)$.

\item Peetre's Theorem \cite{Hormander3,
    Peetre} characterizes the explicit structure of
  $\nabla$--differential operators and shows that such an operator is,
  indeed, a classical differential operator in any coordinate
  chart.

\item
Note that in our approach avoiding local coordinates, one also needs
to consider vector valued Sobolev spaces (with values in tensor
products of the cotangent bundle) even if one is interested only in
scalar equations.
\end{enumerate}
\end{remark}

We have the following ``easy'' continuities.

\begin{lemma}\label{lemma.easy.cont}
The definition of $\nabla$-differential operators gives for $k \in \NN$:
\begin{enumerate}
\item $\nabla : W_{\nabla}^{k+1, p}(M; E) \to W_{\nabla}^{k, p}(M;
  T^*M \otimes E)$ is continuous.
\item \label{item.easy.cont} For $P$ as in Definition
  \ref{def.diff.op} with coefficients in $W^{\ell, \infty}$, we have
  that
  \begin{equation}\label{eq.cont.diff}
   P \seq \sum_{j=0}^\mu a_j \nabla^j \colon {W_{\nabla}^{k+\mu ,
       p}}(M; E) \, \to \, { W_{\nabla}^{k, p}}(M; E),\ \ 0 \leq k
   \leq \ell\,,
  \end{equation}
  is well-defined and continuous.
\end{enumerate}
\end{lemma}

\begin{proof}
For all $j \le k$ and $u \in W_{\nabla}^{k+1, p}(M; E)$, we have
$\nabla^j \nabla u \in L^2(T^{^\otimes (j+1)} M \otimes E)$, by
definition, and hence $\nabla u \in W_{\nabla}^{k, p}(M; E)$, again by
definition. The continuity in the second part follows by combining the
first part with Proposition \ref{prop-Sobolev}, which gives the
continuity of the maps
\begin{equation*}
    W_{\nabla}^{k, \infty}(M; \Hom( TM^{*\otimes j} \otimes E;
      F) ) \otimes W_{\nabla}^{k,p}(M; T^*M^{\otimes j}\otimes E) \,
    \to \, W_{\nabla}^{k, p}(M; F)\, .
\end{equation*}
In particular, $P$ is well-defined with the stated domain and range.
\end{proof}

We want to extend this mapping property to other spaces. It extends to
$k \ge 0$ real immediately by interpolation, since our fractional
order Sobolev spaces were defined by interpolation, see, for instance,
\cite{LionsMagenes1, Taylor1}, or \cite[Chapter 2]{TriebelBook}.
To extend also to $k \le 0$, we shall need also the following basic
algebraic properties, whose statements rely on the following
notation.

\begin{notation}\ref{not.Diff}
Recall that $\Diff^{\mu}(M; E, F)$is the set of $\nabla$--differential
operators $\CIc(M; E) \to \CIc(M; F)$ of order $\le \mu$ with smooth
coefficients. (On the rare occasions when we shall need to show the
dependence on the connection, we shall also write
$\Diff_{\nabla}^{\mu}(M; E, F)$ for this space). Similarly,
$\Diff_b^{\mu}(M; E, F) \subset \Diff^{\mu}(M; E, F)$ is the set of
$\nabla$--differential operators of order $\le \mu$ with
$\CIb$-coefficients.  We shall also write:
\begin{itemize}
\item $\Diff^\mu(M; E) := \Diff^\mu(M; E,
  E)$ and $\Diff_b^\mu(M; E) := \Diff_b^\mu(M; E, E)$;

\item $\Diff^\infty(M; E) = \cup_\mu \Diff^\mu(M; E)$ and
  $\Diff_b^\infty(M; E) = \cup_\mu \Diff_b^\mu(M; E)$, which will be
  seen to be algebras in the next proposition.

\item We will omit $E$ from the notation when $E = F = \CC$. Thus
  $\Diff^\mu(M) := \Diff^\mu(M; \CC) = \Diff^\mu(M; \CC, \CC)$ and
  $\Diff_b^\mu(M) := \Diff_b^\mu(M; \CC) = \Diff_b^\mu(M; \CC, \CC)$,
  and so on.
\end{itemize}
\end{notation}

\begin{proposition}\label{prop.diff.algebra}
We use the notation introduced in Definition \ref{def.diff.op} and
\ref{not.Diff}.
\begin{enumerate}[(i)]
\item The spaces $\Diff^{\mu}(M; E, F)$ and $\Diff_b^{\mu}(M; E, F)$
  are linear vector spaces.

\item \label{item.comp.prop} Let $P \in \Diff^j(M; E, F)$ and $Q \in
  \Diff^N(M; F, G)$, then $QP \in \Diff^{N+j}(M; E, G)$.

\item If $P$ and $Q$ have $\CIb$--coefficients, then $QP$ has
  $\CIb$-coefficients as well.

\item \label{item.pda.iv} In particular,
\begin{equation*}
  \Diff^\infty(M; E) \ede \bigcup_\mu \Diff^\mu(M; E, E) \mbox{ and }
  \Diff_b^\infty(M; E) \ede \bigcup_\mu \Diff_b^\mu(M; E, E)
\end{equation*}
are algebras.
\end{enumerate}
\end{proposition}

\begin{proof}
The statement (i) follows right away from the definitions of the
spaces $\Diff_b^{\mu}(M; E, F) \subset \Diff^{\mu}(M; E, F)$ and the
fact that the spaces of $\CIb \subset \CI$-sections (over $M$ of
various vector bundles) are vector spaces.

By linearity (part (i)), it is enough to assume that $Q = b \nabla^j$
and $P = a \nabla^N$, where $a$ and $b$ are suitable sections of
endomorphism bundles, as in Definition \ref{def.diff.op}. We shall
prove the statements (ii) and (iii) by induction on $j$.

Let us prove the statement (ii). If $j = 0$, this is true since $QP =
ba \nabla^N$ and $ba$ is a smooth endomorphism. Moreover, Remark
\ref{rem.better} gives
\begin{equation}
   \label{eq.product}
   (b \nabla^j) \circ (a \nabla^N) (u) \seq b \nabla^{j-1} \big
         [(\nabla a)\nabla^N u + (1 \otimes a)\nabla^{N+1}u \big ]\,,
\end{equation}
so the first statement follows by induction on $j$, since $(\nabla
a)\nabla^N + (1 \otimes a)\nabla^{N+1}$ is also a
$\nabla$-differential operator with smooth coefficients (if $a \in
\CI$, then $\nabla (a)$ and $1 \otimes a$ are also in $\CI$).

For the statement (iii), let us assume that $P$ and $Q$ have
$\CIb$--coefficients and prove that $QP$ has the same property. We
then proceed in exactly the same way, using the formula
\eqref{eq.product}. So assume $a, b \in \CIb$. When $j = 0$, we obtain
similarly that $QP = ba\nabla^N$ and $ba$ is in $\CIb$ due to
Proposition \ref{prop-Sobolev}. For $j > 0$, we notice that $(\nabla
a)\nabla^N + (1 \otimes a)\nabla^{N+1}$ has $\CIb$ coefficients (if $a
\in \CIb$, then $\nabla (a)$ and $1 \otimes a$ are also in $\CIb$),
which yields the induction step.

The statement (iv) follows right away from the previous two.
\end{proof}

We can now address the dependence of the definitions of
$\nabla$-Sobolev spaces and $\nabla$-differential operators on
$\nabla$.

\begin{remark}
The composition property proved in Proposition
\ref{prop.diff.algebra}\eqref{item.comp.prop} gives that, in
particular, if we are given two connections, $\widetilde \nabla$ and
$\nabla$ on $E$ (so that $\widetilde \nabla - \nabla =: A \in
\CI(M; \End(E))$), then $\Diff_{\nabla}^\mu(M; E, F) =
\Diff_{\widetilde \nabla}^\mu(M; E, F)$, which justifies dropping the
index $\nabla$. The analogous equality $\Diff_{b, \nabla}^\mu(M; E, F)
= \Diff_{b, \widetilde \nabla}^\mu(M; E, F)$ for operators with
$\CIb$-coefficients is true if, and only if, $A \in
\CIb(M; \End(E))$. This is stated as part of the following
proposition.
\end{remark}

We can now prove the independence on the connection $\nabla$ of our
definitions, under certain conditions.

\begin{proposition}\label{prop.ind.A}
Let $A \in \CI(M; \End(E))$, $A^* = -A$, and $\widetilde \nabla =
\nabla + A$. Let also $p \in [1, \infty]$ and $\mu \in \NN$. We then
have the following:
\begin{enumerate}[(i)]

\item $A \in W_{\nabla}^{\infty, \infty}(M; \End(E))$ if, and only if,
  $\Diff_{b, \nabla}^\mu(M; E, F) = \Diff_{b, \widetilde
  \nabla}^\mu(M; E, F)$.

\item If $A \in W_{\nabla}^{\ell-1, \infty}(M; \End(E))$, then there
  is $C_\ell \ge 1$ that depends only on
  $\|A\|_{W_{\nabla}^{\ell-1,\infty }(M)}$ such that, for all $u \in
  \CIc(M; E)$,
\begin{equation*}
    C_\ell^{-1} \|u\|_{W_{\nabla}^{\ell, p}(M)} \le
    \|u\|_{W_{\widetilde \nabla}^{\ell, p}(M)} \le C_\ell
    \|u\|_{W_{\nabla}^{\ell, p}(M)}\,.
\end{equation*}
Consequently,
\begin{equation*}
    W_{\nabla}^{\ell, p}(M; E) \seq W_{\widetilde \nabla}^{\ell, p}(M;
    E).
\end{equation*}
\end{enumerate}
\end{proposition}

\begin{proof}
The first statement follows from Proposition \ref{prop.diff.algebra}.
Indeed, the third statement of this proposition combined with the
assumption that $A \in C_{b}^{\infty }(M; \End(E))$ shows that the
operator $(\nabla +A)^j$ has $\CIb$--coefficients, for all $0\leq
j\leq \mu $. Moreover, any operator $P$ in $\Diff_{b, \widetilde
  \nabla}^\mu(M; E, F)$ has the form $P =\sum_{j=0}^{\mu}\widetilde
a^{[j]}\widetilde\nabla^j$, where $\widetilde a^{[j]}$ are
$\CIb$--coefficients. Thus, $P$ can be written in the equivalent form
$P=\sum_{j=0}^{\mu}\widetilde a^{[j]}(\nabla
+A)^j=\sum_{j=0}^{\mu}b^{[j]}\nabla ^j$, where $b^{[j]}$ are
$\CIb$--coefficients as well. This shows the inclusion $\Diff_{b,
  \widetilde\nabla}^\mu(M; E, F)\subseteq \Diff_{b, \nabla}^\mu(M; E,
F)$. The converse inclusion follows by symmetry.

Next we assume that $\Diff_{b, \nabla}^\mu(M; E, F) = \Diff_{b,
  \widetilde \nabla}^\mu(M; E, F)$. Then the operators $\nabla$ and
$\widetilde \nabla =\nabla +A$ belong to $\Diff_{b, \nabla}^\mu(M; E,
F)$, and, thus, $A = \widetilde \nabla - \nabla \in C_{b}^{\infty }(M; \End(E))$,
since the later is a vector space.

The second statement is proved by induction on $\ell$ using
Propositions \ref{prop-Sobolev} and \ref{prop.diff.algebra}.  Indeed,
first of all, if $\ell =0$ then $W_{\nabla}^{0, p}(M; E)=L^p(M; E) =
W_{\widetilde\nabla}^{0, p}(M; E)$ and the property follows, with $C_0
=1$.

Assume now that the property we want to prove holds for $\ell - 1 \ge
0$ and show it for $\ell$. The definition of the norm in
$W_{\nabla}^{\ell , p}(M;E)$ and the induction hypothesis imply that
\begin{align*}
    \|u\|_{W_{\nabla}^{\ell , p}(M)}^p & \le \|u\|_{W_{\nabla}^{\ell
        -1,p}(M)}^p + \|\nabla u\|_{W_{\nabla}^{\ell -1,p}(M)}^p \\
    & = \|u\|_{W_{\nabla}^{\ell -1,p}(M)}^p + \|\widetilde \nabla u -
    A u\|_{W_{\nabla}^{\ell -1,p}(M)}^p \\
    & \leq C_{\ell-1}^p \|u\|_{W_{\widetilde\nabla}^{\ell -1,p}(M)}^p
    + 2^{p-1} \|\widetilde \nabla u\|_{W_{\nabla}^{\ell -1,p}(M)}^p +
    2^{p-1}\| A u \|_{W_{\nabla}^{\ell -1,p}(M)}^p \\
    & \leq C_{\ell-1}^p (1 + 2^{p-1}) \| u \|_{W_{\widetilde
        \nabla}^{\ell, p}(M)}^p + 2^{p-1}C_{\ell -1 , p}^p \|
    A\|^p_{W_{\nabla}^{\ell -1,\infty}(M)} \| u \|_{W_{\nabla}^{\ell
        -1,p}(M)}^p \\
    & \leq C_{\ell-1}^p 2^{p-1} \Big [2+ C_{\ell -1 , p}^p \|
      A\|^p_{W_{\nabla}^{\ell -1,\infty}(M)} \Big ] \| u
    \|_{W_{\widetilde \nabla}^{\ell,p}(M)}^p \,,
\end{align*}
where $C_{\ell-1, p}$ is the (absolute) constant of Proposition
\ref{prop-Sobolev}. The first desired inequality for the norm then
follows if $C_{\ell-1}^p 2^{p-1} \Big [2+ C_{\ell -1 , p}^p \|
  A\|^p_{W_{\nabla}^{\ell -1,\infty}(M)} \Big ] =:
C_{\ell}^{p}$. Since we can bound $\| A\|^p_{W_{\widetilde
    \nabla}^{\ell -1,\infty}(M)}$ in terms of $\|
A\|^p_{W_{\nabla}^{\ell -1,\infty}(M)}$, by the induction hypothesis,
the second of the desired inequalities for the norms follows by
symmetry. In particular, we deduce the equality $W_{\nabla}^{\ell,
  p}(M; E) \seq W_{\widetilde \nabla}^{\ell, p}(M;E)$.
\end{proof}

\section{Mixed differential operators and totally bounded vector fields}
\label{sec.3}

In this section we look at a different type of differential
  operators, which we call ``mixed differential operators,'' since
  they will be used to relate the $\nabla$-differential operators of
  the previous section to the classical differential operators.  To
  study them, we introduce the Fr\'echet finiteness condition (FFC)
  and we show that if (FFC) is satisfied, then the mixed differential
  operators coincide with the $\nabla$ ones.

\subsection{The Fr\'echet finite generation condition}
It is known \cite[p. 71]{Milnor-Stasheff} that every vector bundle $E \to M$ on a manifold
has an embedding $\Phi : E \to M \times \RR^N$ into a trivial
bundle. In this subsection, we shall use the existence of such an
embedding for $E = TM$ and we shall deduce some geometric consequences.

\begin{remark}\label{rem.generation1}
Let $\Phi : TM \to M \times \RR^N$ be a smooth embedding of the
tangent bundle into a trivial vector bundle. We endow the trivial
vector bundle with the constant metric. Then the transpose $\Phi^\top
: M \times \RR^N \to TM$ is onto. (The transpose is the analog of the
adjoint, but in the real case.) Moreover, $\Phi^\top \Phi$ is a
smooth, invertible section of $\End(TM)$. Let $\Psi := (\Phi^\top
\Phi)^{-1} \Phi^\top $, so that $\Psi \in \CI(M; \Hom(\RR^N ; TM)$ and
$\Psi \Phi = 1$. Let
\begin{equation*}
  Z_1 \ede \Psi(e_1),\ Z_2 \ede \Psi(e_2),\ \ldots\ ,\ Z_N \ede
  \Psi(e_N) \, \in\, \CI(M; TM)
\end{equation*}
be the vector fields corresponding to the constant basis
$(e_j)_{j=1}^N$ of $\RR^N$ via $\Phi^\top$. Since $\Psi \in \CI(M;
\Hom(\RR^N; TM))$, we have that $Z_j$ are all in $\CI(M; TM)$. Let
\begin{equation*}
  \xi_j \ede p_j \circ \Phi : TM \to \RR
\end{equation*}
be the 1-form obtained from the projection of $\RR^N$ onto the $j$th
component. Then $\xi_j \in \CI(M; T^*M)$ and the relation $\Psi \Phi =
1$ gives, for every $X \in \CI(M; TM)$,
  \begin{equation*}
    X \seq \Psi(\Phi(X)) \seq \Psi \Big (\sum_{j=1}^N \xi_j(X) e_j
    \Big ) \seq \sum_{j=1}^N \xi_j(X) Z_j\,.
  \end{equation*}
In particular,
  \begin{equation*}
    \CI(M; TM) \seq \sum_{j=1}^N \CI(M) Z_j\,.
  \end{equation*}
Let $\omega \in \CIb(M; T^*M)$. By evaluating $\omega$ in the above
relation, we obtain that $\omega(X) = \sum_{j=1}^N \xi_j(X)
\omega(Z_j)$, and hence we have the dual relation $\omega \seq
\sum_{j=1}^N \omega(Z_j) \xi_j\,.$
\end{remark}

We now take a look at a different global definition of differential
operators.

\begin{definition}\label{def.diff.mixed}
We let $\widetilde \Diff^\mu(M; E, F)$ to be the set of all linear
operators of order $\le \mu$ linearly generated by $a \nabla_{X_1}
\ldots \nabla_{X_r}$, $0 \le r \le \mu$, $a \in \CI(M; \Hom(E, F))$,
$X_j \in \CI(M; TM)$. An operator $P$ of this type will be called a {\em
  mixed differential operator of order $\le \mu$.} If $a \in \CIb(M;
\Hom(E, F))$ and all $X_j \in \maW_b(M) := \CIb(M; TM)$, then $P$ is
called a a {\em mixed differential operator of order $\le \mu$ with
  $\CIb$-coefficients}, and the set of all such operators is denoted by $\widetilde
\Diff_b^\mu(M; E, F)$.
\end{definition}

As we will see in Subsection \ref{sec.4}, mixed differential operators
are, sometimes, easier to deal with than $\nabla$-differential
operators and form a convenient intermediate class between $\nabla$-
and classical differential operators. The following lemma is standard,
except for the fact that the system $(Z_j)$ is only a system of
generators of $\CI(M; TM)$ as a $\CI(M)$-module, and not a basis.

\begin{lemma}\label{lemma.just.smooth}
Let $Z_j$ and $\xi_j$ be as in Remark \ref{rem.generation1}.
\begin{enumerate}[(i)]
   \item If $X \in \CI(M; TM)$, then $\nabla_X \in \Diff^{1}(M; E,
     T^*M \otimes E)$.

   \item\label{item.FFC+div.i} $\nabla^E : \CIc(M; E) \to \CIc(M; T^*M
     \otimes E)$ satisfies $\nabla^E = \sum_{j=1}^N \xi_j \otimes
     \nabla^E_{Z_j}$.

    \item \label{item.div} For any $X \in \CI(M; TM)$, we have
      $\div(X) \in \CI(M)$.

    \item We have $\nabla_{Z_i}Z_j= \sum_{k} G_{ij}^k Z_k$, where
      $G_{ij}^k \in \CI(M)$.

    \item We have $[Z_i, Z_j] = \sum_{k} L_{ij}^k Z_k$, where
      $L_{ij}^k \in \CI(M)$.
\end{enumerate}
\end{lemma}

\begin{proof}
Let $Z_1, Z_2, \ldots, Z_N \in \CI(M; TM)$ be the vector fields
introduced in Remark \ref{rem.generation1}.

To prove (i), let $i_X : T^*M \otimes E \to E$ be the contraction with
the vector $X \in \CI(M ;TM)$. Then $i_X \in \CI(M; \Hom(T^*M \otimes E;
E))$ since $X \in \CI(M; TM)$, and hence $\nabla_X = i_X \circ \nabla$ is
a $\nabla$-differential operator with smooth coefficients, by
definition.

To prove (ii), let $u \in \CIc(M; E)$ and $X \in \CIc(M; TM)$ be
arbitrary.  Let $i_X : T^*M \otimes E \to E$ be the contraction with
the vector $X \in \CIc(M; TM)$, $i_X(u) = \langle X, u \rangle$. Then
the formula $X = \sum_{j=1}^N \xi_j(X) Z_j$ of Remark
\ref{rem.generation1} gives
  \begin{equation*}
    \langle X, \nabla^E(u) \rangle \seq \nabla^E_X (u) \seq
    \sum_{j=1}^N \xi_j(X) \nabla^E_{Z_j}(u) \seq \Big \langle X,
    \sum_{j=1}^N \xi_j \otimes \nabla^E_{Z_j}(u) \Big \rangle\,.
  \end{equation*}
Since $u \in \CIc(M ; T^*M)$ and $X \in \CIc(M; TM)$ were arbitrary,
the result follows.

To prove (iii), let $X_1, X_2, \ldots, X_n$ be a local {\em
  orthonormal} basis of $TM$. Then we obtain
\begin{align*}
  \div(X) & \seq \tr (\nabla X)\\
  & \seq \sum_{i=1}^n (\nabla_{X_i}X , X_i)\\
  & \seq \sum_{j=1, k, l }^n \xi_k (X_j) \xi_l(X_j) (\nabla_{Z_k}X ,
  Z_l)\\
  & \seq \sum_{j=1, k, l }^n (\xi_k, \xi_l) (\nabla_{Z_k}X ,
  Z_l)
\end{align*}
since the map $T^*M \ni \xi \to \big ( (\xi, X_j) \big)_{j=1}^n$ is an
isometry (since $(X_i)$ was chosen to be an orthonormal basis). Hence, the result follows.

To prove (iv), we notice that, since $\nabla_{Z_i}Z_j \in \CI(M; TM)$,
we have $\nabla_{Z_i}Z_j = \sum_{k=1}^N \xi_k(\nabla_{Z_i}Z_j) Z_k$,
by Remark \ref{rem.generation1}, and hence we can take $G_{ij}^k :=
\xi_k(\nabla_{Z_i}Z_j)$, which is in $\CIb(M)$ by Proposition
\ref{prop-Sobolev}. For (v), we proceed in exactly the same way since
the Lie bracket of vector fields $[Z_i, Z_j] \in \CI(M; TM)$.
\end{proof}

We have the following generation property for mixed differential
operators.

\begin{proposition}
\label{prop.inclusion.mixed1}
Let $\mu ,\nu \in \NN$.
\begin{enumerate}[(i)]
\item $\widetilde \Diff^{\nu}(M; F, G) \widetilde \Diff^\mu(M;
  E, F) \subset \widetilde \Diff^{\mu + \nu}(M; E, G)$.

\item $\widetilde \Diff^\mu(M; E, F) = \Diff^\mu(M; E, F)$.

\item Let $Z_1, Z_2, \ldots, Z_N \in \CI(M)$ be a systems of
  generators for $\CI(M)$ as in \ref{rem.generation1}, then
  $\Diff^\mu(M; E, F)$ is linearly generated by $a
  \nabla_{X_1}\nabla_{X_2} \ldots \nabla_{X_r}$, where $r \le \mu$,
  $X_1, X_2, \ldots, X_r \in \{Z_1, Z_2, \ldots, Z_N\}$ and $a \in
  \CI(M; \Hom(E, F))$.
\end{enumerate}
\end{proposition}

\begin{proof}
We prove the first statement by induction on $\mu$. For $\mu = 0$, the
property (i) follows from the multiplication properties
\begin{equation*}
  \CI(M; \Hom(F; G)) \CI(M; \Hom(E; F)) \subset \CI(M; \Hom(E; G))
  \,.
\end{equation*}
The induction step is obtained using also the equation $\nabla_{X} a =
\nabla_X(a) + a \nabla_X$.

To prove (ii), let us notice that Lemma \ref{lemma.just.smooth}(i)
states that $\nabla_X \in \Diff^1(M; E, F)$. The composition property
of Proposition \ref{prop.diff.algebra}\eqref{item.comp.prop}, yields
the inclusion  $\widetilde \Diff^\mu(M; E, F) \subset \Diff^\mu(M; E, F)$.
Let us prove now the opposite inclusion.

Let $\tau_\xi(\zeta) := \xi \otimes \zeta$ and $\nabla \seq
\sum_{i=1}^N \tau_{\xi_i}\nabla_{Z_i},$ as in Lemma
\ref{lemma.just.smooth}. This gives $\nabla \in
\widetilde{\Diff}^1(M; E, T^*M \otimes E)$. Part (i), already
proved, then proves, by induction on $j$, that $a \nabla^j \in
\widetilde \Diff^\mu(M; E, F)$ if $a \in \CI(M; \Hom(E; F))$ and $j
\le \mu$. Therefore $\Diff^\mu(M; E, F) \subset \widetilde
\Diff^\mu(M; E, F)$. Hence we have equality.

Let $\maD_\mu$ be the linear span of $a \nabla_{X_1}\nabla_{X_2}
\ldots \nabla_{X_r}$, where $r \le \mu$ and $X_1, X_2, \ldots, X_r \in
\{Z_1, Z_2, \ldots, Z_N\}$ and $a \in \CI(M; \Hom(E, F))$. It is
enough to prove that $\maD_\mu = \Diff^\mu(M; E, F)$.

Let $Q := a \nabla_{X_1}\nabla_{X_2} \ldots \nabla_{X_r}$, where $r
\le \mu$ and $X_1, X_2, \ldots, X_r \in \CI(M)$ and $a \in \CI(M;
\Hom(E, F))$. We shall prove, by induction on $r$, that $Q \in
\maD_\mu$, that is, that $Q$ is a linear combination of terms of the
same kind, but with all $X_j \in \{Z_1, Z_2, \ldots, Z_N\}$. By
induction, we can assume that this is true for products of up to $r -
1$ covariant derivatives. The induction step is obtained using the
equation $\nabla_{X} a = \nabla_X(a) + a \nabla_X$. Indeed, let us
consider then $a$ and $b$ be $\CIb$--endomorphisms, $X_2, X_3, \ldots,
X_r \in \{Z_1, Z_2, \ldots, Z_N\}$ and $X_1$ an arbitrary smooth
vector field. Then
\begin{equation*}
  Q_1 \ede a \nabla_{X_1} b \nabla_{X_2} \ldots \nabla_{X_r} \seq a
  \big [X_1(b) + b \sum_{j=1}^N \xi_j(X_1) \nabla_{Z_j} \big ]
  \nabla_{X_2} \ldots \nabla_{X_r}
\end{equation*}
We thus have $Q_1 \in \maD_\mu$. This proves the equality of
$\maD_\mu$ and $\widetilde{\Diff}^\mu(M; E, F)$.
\end{proof}

We obtain the following consequence.

\begin{corollary}\label{cor.order1}
Let $(Z_j)$, $1 \le j \le N$, $Z_j \in \CI(M; TM)$ be as in Remark
\ref{rem.generation1}. Then $\Diff^\mu(M; E, F)$ is linearly generated
by $a \nabla^E_{Z_{k_1}} \ldots \nabla^E_{Z_{k_r}}$, where $1 \le k_1
\le k_2 \le \ldots \le k_r \le N$, $r \le \mu$, and $a \in \CI(M;
\Hom(E; F))$.
\end{corollary}

\begin{proof}
We know that $\Diff^\mu(M; E, F)$ is linearly generated by terms of
the form $a \nabla^E_{Z_{k_1}} \ldots \nabla^E_{Z_{k_r}}$, where $k_1,
k_2, \ldots, k_r \in \{1, 2, \ldots, N\}$, $r \le \mu$, $a \in \CI(M;
\Hom(E; F))$. It just remains to show that we can choose the indices
$k$ to form a non-decreasing sequence. To this end, we shall use the
relation
\begin{equation}\label{eq.def.R}
    R^E(X, Y) \ede [\nabla_X^E, \nabla_Y^E]-\nabla_{[X, Y]}^E
   \in \CI(M; \End(E))\,.
\end{equation}
This shows, that, up to lower order terms, we can commute
the operators $\nabla_{Z_j}^E$. The proof is completed by induction.
\end{proof}

\subsection{Totally bounded vector fields and the (FFC) condition}

Recall that we write $\Diff_b^\infty(M) := \Diff_b^\infty(M; \CC)$ for
the algebras of differential operators on $M$ with
$\CIb(M)$-coefficients introduced in Proposition
\ref{prop.diff.algebra}\eqref{item.pda.iv} when $E$ is the trivial
vector bundle with fiber $\CC$. Then, by sepparating the order zero
part of a differential operator, we obtain
\begin{equation*}
    \Diff^1_b(M) \seq \CIb(M) \oplus \maW_b(M)\,,
\end{equation*}
where
\begin{equation*}
  \maW_b(M) \ede \CIb(M; TM) \ede W^{\infty, \infty}(M; TM)\,,
\end{equation*}
that is, the space of bounded vector fields on $M$ all of whose
covariant derivatives are bounded. Then $\maW_b(M)$ is a module over
$\CIb(M)$ with respect to multiplication by Proposition
\ref{prop-Sobolev}. This space (which will turn out to be a Lie
algebra) will play an important role in what follows and this section
is devoted, to a large extent, to the study of their role in the
definition of $\nabla$-Sobolev spaces and $\nabla$-differential
operators.

\begin{lemma}\label{lemma.WbM}
Let $\maW_b(M) \ede \CIb(M; TM)$
\begin{enumerate}[(i)]
     \item If $X \in \maW_b(M)$, then $\nabla_X \in \Diff_b^{1}(M; E,
       T^*M \otimes E)$.

   \item $\maW_b(M)$ is a Lie algebra, that is, $[X, Y] := XY - YX \in
     \maW_b(M)$ for all $X, Y \in \maW_b(M)$.

   \item If $X, Y \in \maW_b(M)$, then $\nabla_X Y \in \maW_b(M)$.
\end{enumerate}
\end{lemma}

\begin{proof}
To prove (i), let $i_X : T^*M \otimes E \to E$ be the contraction
with the vector $X \in \maW_b(M)$. Then $i_X \in W^{\infty,
\infty}(M; \Hom(T^*M \otimes E; E))$ since $X \in \maW_b(M)$, and
hence $\nabla_X = i_X \circ \nabla$ is a $\nabla$-differential
operator with coefficients in $\CIb$, by definition. The property
(ii) follows from (i) since $\Diff_b^\infty(M) = \Diff_b^\infty(M;
\CC)$ is an algebra (see Proposition \ref{prop.diff.algebra}\eqref{item.pda.iv}).
Property (iii) follows from (i) and the ``easy'' mapping property of Lemma
\ref{lemma.easy.cont} (which gives, in particular, that $P \in
\Diff_b^{\mu}(M; E, F)$ maps $W_\nabla^{\infty, \infty}(M; E)$ to
$W_\nabla^{\infty, \infty}(M; F)$ continuously.
\end{proof}

\begin{definition}\label{def.FFC}
We say that $(M, g)$ satisfies the {\em Fr\'echet finiteness condition
  (FFC)} if, there exists $N \in \NN$ and an isometric (vector bundle)
embedding $\Phi : TM \subset M \times \RR^N$, $\Phi \in W^{\infty,
  \infty}(M; \Hom(TM; \RR^N))$, where on $\RR^N$ we consider the flat
connection.
\end{definition}

This condition may seem strong, but it was proved in Lemma 3.1 of
\cite{GN17} that it is satisfied by a manifold with bounded
geometry. In fact, in that paper, instead of constructing a $\Phi$
with the property that it is an isometry, it was proved that
$\Phi^{-1} \in \CIb(M; \Hom(\RR^N; TM))$. By replacing that $\Phi$
with its polar part, one can assume it to be isometric.  In
particular, the (FFC) property is {\em hereditary}, in the sense that
if it is satisfied by a manifold $M$, then it is satisfied by any open
subset $M_0 \subset M$. Indeed, it is enough to restrict $\Phi$ to
$M_0$.

We have the following analog of Remark \ref{rem.generation1}.

\begin{remark}\label{rem.generation2}
We use the notation of Remark \ref{rem.generation1}.  Since we have
assumed that $\Phi$ is an isometry now, we have $\Psi = \Phi^\top : M
\times \RR^N \to TM$ and, for all $1 \le j \le N$, we have
\begin{equation*}
  \begin{gathered}
    Z_j \ede \Phi^\top(e_j) \in \maW_b(M) \ede \CIb(M; TM) \ \mbox{
      and }\\
    \xi_j \ede p_j \circ \Phi \in \CIb(M; T^*M)
  \end{gathered}
\end{equation*}
so that $\maW_b(M) = \sum_{j=1}^N \CIb(M) Z_j$.
The relations
  \begin{equation*}
    X \seq \sum_{j=1}^N \xi_j(X) Z_j\ \mbox{ and } \
    \omega \seq \sum_{j=1}^N \omega(Z_j) \xi_j\
  \end{equation*}
remain, of course, valid.  The set $\{Z_1, Z_2, \ldots, Z_N\}$ will be
called a {\em Fr\'echet system of generators} for $\maW_b(M)$.
\end{remark}

Let $\div := - d^\prime$ be the negative of the dual map
$d^\prime : \CIc(M; TM) \to \CIc(M)$ of $d : \CIc(M) \to \CIc(M;
T^*M)$. We shall need the following result.

\begin{lemma}\label{lemma.FFC+div}
Assume $M$ satisfies the (FFC)
condition and let $\xi_j \in \CIb(M; T^*M)$ and $Z_j \in \maW_b(M)
:= \CIb(M; TM)$, $j = 1, \ldots, N$, be as in Remark
\ref{rem.generation2}.
\begin{enumerate}[(i)]

   \item If $X \in \maW_b(M)$, then $\nabla_X \in \Diff_b^{1}(M; E,
     T^*M \otimes E)$.

    \item \label{item.div} For any $X \in \maW_b(M)$,
      we have $\div(X) \in \CIb(M)$.

    \item {We have $\nabla_{Z_i}Z_j= \sum_{k} G_{ij}^k Z_k$,
      where $G_{ij}^k \in W^{\infty,\infty}(M)$.}

    \item {We have $[Z_i, Z_j] = \sum_{k} L_{ij}^k Z_k$, where
      $L_{ij}^k \in W^{\infty,\infty}(M)$.}
\end{enumerate}
\end{lemma}

\begin{proof}
To prove (i), let $u \in \CIc(M; E)$ and $X \in \CIb(M; TM)$ be
arbitrary. Let $i_X : T^*M \otimes E \to E$ be the contraction with
the vector $X \in \maW_b(M)$, $i_X(u) = \langle X, u \rangle$. Then
the second displayed formula in the last remark gives
\begin{equation*}
\langle X, \nabla^E(u) \rangle \seq \nabla^E_X (u) \seq
\sum_{j=1}^N \xi_j(X) \nabla^E_{Z_j}(u) \seq \langle X,
\sum_{j=1}^N \xi_j \otimes \nabla^E_{Z_j}(u) \rangle\,.
\end{equation*}
Since $u \in \CIc(M ; E)$ and $X \in \CIb(M; TM)$ were arbitrary, the
result follows.

Let $X_1, X_2, \ldots, X_n$ be a local {\em orthonormal} basis of $TM$.
Let $Z_1, Z_2, \ldots, Z_N \in \maW_b(M) := W^{\infty, \infty}(M; TM)$ be
the Fr\'echet system of generating vector fields introduced in Remark
\ref{rem.generation2}.  Then we obtain
\begin{align*}
  \div(X)
  & \seq \sum_{i=1}^n (\nabla_{X_i}X , X_i)\\
  & \seq \sum_{j=1, k, l }^n \xi_k (X_j) \xi_l(X_j) (\nabla_{Z_k}X ,
  Z_l)\\
  & \seq \sum_{j=1, k, l }^n (\xi_k, \xi_l) (\nabla_{Z_k}X ,
  Z_l)
\end{align*}
by the isometry property of the map $T^*M \ni \xi \to \big ( (\xi, X_j) \big)_{j=1}^n$
(since $(X_i)$ was chosen to be an orthonormal basis).

Lemma \ref{lemma.WbM} gives that $\nabla_{Z_i}Z_j \in \maW_b(M)$.
Hence $\nabla_{Z_i}Z_j = \sum_{k=1}^N \xi_k(\nabla_{Z_i}Z_j) Z_k$ and
we can take $G_{ij}^k := \xi_k(\nabla_{Z_i}Z_j)$, which is in
$\CIb(M)$ by Proposition \ref{prop-Sobolev}. For (iv), we proceed in
exactly the same way by using Lemma \ref{lemma.WbM}(iii).
\end{proof}

We can now formulate and prove the following proposition, which
provides us with the usual properties of the Hilbert space adjoints
$\nabla_X^*$ and $\nabla^*$.

\begin{proposition}
\label{mapping-properties}
Let $\mu \in \NN$ and $1<p<+\infty$ and let us assume that $M$
satisfies the Fr\'echet finiteness condition (Definition
\ref{def.FFC}). Then we also have the following properties.
\begin{enumerate}[(i)]
  \item If $X \in \maW_b(M)$, then $\nabla_X^* = - \nabla_X -
    \div(X) \in \Diff_b^{1}(M; T^*M \otimes E, E)$.
  \item $\nabla^* \in \Diff_b^{1}(M; T^*M \otimes E, E)$.
\end{enumerate}
\end{proposition}

\begin{proof}
To prove (i), let $\div = -d^\prime : \CIc(M; T^*M) \to \CIc(M)$, the
negative of the transpose of the de Rham differential. Then $\div(fX)
= f\div(X) + X(f)$. Let us write $\nabla_X$ for $\nabla_X^E$.  Then,
using formula \eqref{covariant-derivative}, we obtain for all $\xi,
\eta\in \CIc(M; E)$
\begin{align*}
   (\nabla_X \xi, \eta)_E & \seq X(\xi, \eta)_E - (\xi, \nabla_X
  \eta)_E \\
  & \seq \div\left((\xi, \eta)_E X\right) - (\xi, \eta)_E\div(X) -
  (\xi, \nabla_X \eta)_E\,,
\end{align*}
and then, integrating over M and using the assumption that $\xi$ and
$\eta $ have compact support (so the integral of the ``div'' part is
zero), we get
\begin{align*}
    \int _M(\nabla_X \xi, \eta)_E \dvol & \seq - \int_M (\xi, \div
    (X)\eta +\nabla_X \eta)_E\dvol\,.
\end{align*}
Thus, $\nabla_X^* \seq -\nabla _X - \div(X),$ as stated, and hence
$\nabla_X^*$ belongs to $\Diff_b^{1}(M; T^*M \otimes E, E)$ by Lemma
\ref{lemma.WbM}(i) and by Lemma \ref{lemma.FFC+div}\eqref{item.div}.

Let $\tau_\xi(\zeta) := \xi \otimes \zeta$. Then, using the notation
of Remarks \ref{rem.generation1} and \ref{rem.generation2}, we can
reformulate the result of Lemma
\ref{lemma.FFC+div}\eqref{item.FFC+div.i} as
\begin{equation}\label{eq.reform}
    \nabla \seq \sum_{i=1}^N \tau_{\xi_i}\nabla_{Z_i}\,.
\end{equation}
The relation (ii) thus follows from this relation by taking adjoints
and using (i) and the composition property of Proposition
\ref{prop.diff.algebra}\eqref{item.comp.prop}.
\end{proof}

We can now prove the following extension of the standard continuity of
differential operators.

\begin{corollary}\label{cor.neg.order}
Let $1 < p < \infty$. If $P = \sum_{j=0}^{\mu} a_j \nabla^j $ and
$a_{j} \in \CIb(M; \Hom(E; F))$, that is, if $P$ is a
$\nabla$--differential operator with coefficients in $\CIb$, then $P$
extends by continuity to maps
\begin{equation*}
\begin{gathered}
    W_{0,\nabla }^{s,p}(M; E) \to W_{0,\nabla }^{s-\mu,p}(M; F)\,,
    \quad s\in \NN,\ s \ge \mu, \\
%
   W^{s,p}_\nabla (M; E) \to W_\nabla^{s-\mu,p}(M; F)\,, \quad s \in
   \RR \,.
\end{gathered}
\end{equation*}
\end{corollary}

\begin{proof}
We have already seen that $P:W^{s,p}_\nabla (M; E) \to
W_\nabla^{s-\mu,p}(M; F)$ is continuous if $s \ge \mu$, $s \in \NN$,
see Lemma \ref{lemma.easy.cont}\eqref{item.easy.cont}. Moreover, a
differential operator will send compactly supported sections to
compactly supported sections. The first statement thus follows.

Let us turn now to the general case. Since for non-integer $s$, the
spaces $W^{s, p}_{\nabla}$ are defined by interpolation between
consecutive integers, it suffices to prove our statement for integer
values of $s$. (The general case is obtained by interpolation.)
Furthermore, using Proposition \ref{prop.diff.algebra}, we see that it
is also enough to consider the case $P = \nabla$ (so $\mu = 1$). We
have then two possibilities for $s$, either $s \ge \mu = 1$ or $s \le
0$. The first case was already proved, as we have just mentioned. To
prove the case $s \le 0$, recall that the adjoint operator is also a
$\nabla$--differential operator by Proposition
\ref{mapping-properties} and that the negative order Sobolev spaces
$W^{-s,p}(M; E^*) \ede W_{0,\nabla }^{s,p'}(M; E)^*$, see Equation
\eqref{eq.def.neg.Sob}. As the statement is known for $\nabla^* :
W_{0,\nabla }^{1-s,p'}(M; E)^* \to W_{0,\nabla }^{-s,p'}(M; E)^*$
since $\nabla^* \in \Diff_b^{1}(M; T^*M \otimes E, E)$ by the first
part (since we reduced to $s \in \ZZ_+$), the desired statement is
obtained by taking adjoints.
\end{proof}

We shall use $P$ to denote all the maps in Corollary
\ref{cor.neg.order} induced by the original $P$. We shall not need the
case $s \notin \ZZ$, but that case can usually be handled by
interpolation \cite{LionsMagenes1, TriebelBook}.

Recall that $M$ is an $n$-dimensional manifold with metric $g$ and
$\nabla^{LC}$ is the associated Levi-Civita connection.

\begin{corollary}
\label{diff-oper-weight}
Let $1<p<\infty $. Let $\rho , f_0: M \to (0, \infty)$
be admissible weights with respect to the metric $g_0=\rho ^{-2}g$.
Let $P = \sum_{j=0}^{\mu} a_j \nabla^j $ and $a_{j} \in \CIb(M,g_0;\Hom(E; F))$, that
is $P$ is a $\nabla$--differential operator with coefficients in $\CIb$ with respect to the metric $g_0$.
Then $P$ extends by continuity to maps
\begin{equation}
\label{continuity-weight}
   f_0 W^{\ell , p}_{\nabla, \rho}(M; E) \, \to \,
   f_0 \rho^{-\mu} W^{\ell -\mu , p}_{\nabla, \rho}(M; F)\,,
   \quad \ell \in \NN\,.
\end{equation}
\end{corollary}

\begin{proof}
Recall that the relation between the weighted Sobolev space $f_0 W^{\ell , p}_{\nabla^{LC}, \rho} (M;E)$
(defined with respect to the metric $g$) and the classical Sobolev space $W^{\ell , p}(M,g_0;E)$
(defined with respect to the metric $g_0=\rho ^{-2}g$) is given by the formula
\begin{equation}
\label{weight-classic-1}
  f_0 W^{\ell , p}_{\nabla, \rho} (M;E) \seq f_0
  \rho ^{-\frac{n}{p}}W^{\ell , p}(M,g_0; E)\,.
\end{equation}
In addition, the assumption that $f_0$ is admissible shows that $f_0^{-1}
\rho^{\mu} Pf_0$ is a $\nabla$--differential
operator with coefficients in $\CIb$ with respect to the metric $g_0$, that is, $f_0^{-1}Pf_0$ and $P$ are
both $\nabla$--differential operators of the same type.
Then, in view of Corollary \ref{cor.neg.order}, this operator extends by continuity to the maps
\begin{equation}
\label{map-classic-space}
  W^{\ell , p}(M,g_0;E)\to W^{\ell -\mu,p}(M,g_0;F)\,, \quad \ell \in {\mathbb N}\,.
\end{equation}
Then formula \eqref{weight-classic-1} and the mapping property \eqref{map-classic-space} show that
$P$ extends by continuity to the maps given in \eqref{continuity-weight}.
\end{proof}

\subsection{Bounded mixed differential operators}
\label{sec.4}
The (FFC) condition gives the following finite generation property for
the algebra of $\nabla$-differential operators with
$\CIb$-coefficients. It is analogous to Proposition
\ref{prop.inclusion.mixed1}.

\begin{proposition}
\label{prop.inclusion.mixed2}
Let $\mu \in \NN$. Assume $M$ satisfies (FFC) and let $Z_1, Z_2,
\ldots, Z_N \in \CI(M)$ be a Fr\'echet systems of generators for
$\CI(M)$ as in \ref{rem.generation1},
\begin{enumerate}[(i)]
\item $\widetilde {\Diff}_b^{\nu}(M; F, G) \widetilde \Diff_b^\mu(M;
  E, F) \subset \widetilde \Diff_b^{\mu + \nu}(M; E, G)$.

\item $\widetilde \Diff_b^\mu(M; E, F) = \Diff_b^\mu(M; E, F)$.

\item then $\Diff_b^\mu(M; E, F)$ is linearly generated by $a
  \nabla_{X_1}\nabla_{X_2} \ldots \nabla_{X_r}$, where $r \le \mu$,
  $X_1, X_2, \ldots, X_r \in \{Z_1, Z_2, \ldots, Z_N\}$ and $a \in
  \CIb(M; \Hom(E, F))$.
\end{enumerate}
\end{proposition}

\begin{proof}
The proof is completely similar to that of Proposition
\ref{prop.inclusion.mixed1}, by using the composition property
\begin{equation*}
\CIb(M; \Hom(F; G)) \CIb(M; \Hom(E; F)) \subset \CIb(M; \Hom(E; G))\,,
\end{equation*}
and the properties $\nabla_X \in \Diff_b^1(M; E, F)$ if $X \in \maW_b(M) := \CIb(M;
TM)$, $\nabla \in
\widetilde{\Diff}_m^1(M; E, T^*M \otimes E)$, and
$\nabla_{X} a = \nabla_X(a) + a \nabla_X \in \CIb(M; \Hom(E, F))$
if $a \in \CIb(M; \Hom(E, F))$ and $X \in \maW_b(M)$ instead
of the corresponding statements in that proof.
\end{proof}

Let us recall that a vector bundle $E \to M$ is said to have {\em
  totally bounded curvature} if its curvature $R^E \in \CIb(M;
\Lambda^2 T^*M \otimes \End(E))$. Recall that if $M$ satisfies the
(FFC) condition, then $\widetilde \Diff^\mu = \Diff^\mu$ and
$\widetilde \Diff_b^\mu = \Diff_b^\mu$. We next show that we can
choose the vector fields $X_j$ in the above proposition in the right
order.

\begin{corollary}\label{cor.order2}
Assume $M$ satisfies the (FFC) condition and let $(Z_j)$, $1 \le j \le
N$, be a Fr\'echet generating system for $\maW_b(M)$.  Let us assume
also that $E \to M$ has totally bounded curvature
\begin{enumerate}[(i)]
\item If $X, Y \in \maW_b(M)$, then $\nabla_X^E \nabla_Y^E -\nabla_Y^E
  \nabla_X^E - \nabla_{[X, Y]}^E \in \CIb(M; \End(E))$.

\item Consequently, $\Diff_b^\mu(M; E, F)$ is linearly generated by $a
  \nabla^E_{Z_{k_1}} \ldots \nabla^E_{Z_{k_r}}$, where $1 \le k_1 \le
  k_2 \le \ldots \le k_r \le N$, $r \le \mu$, and $a \in \CIb(M;
  \Hom(E; F))$.
\end{enumerate}
\end{corollary}

\begin{proof}
The statement (i) follows also from the formula \eqref{eq.def.R}
taking into account that, in this case, $R^E(X, Y) \in
\CIb(M; \End(E))$, since $R^E \in \CIb(M; \Lambda^2 T^*M
\otimes \End(E))$ and $X, Y \in \maW_b(M)$.

Finally, the last part is proved in the same way Corollary
\ref{cor.order1}, but taking into account also (i).
\end{proof}

\subsection{Equivalent definitions of $\nabla$-Sobolev spaces}

Let $Z_1, Z_2, \ldots, Z_N \in \maW_b(M) := W^{\infty,
  \infty}(M; TM)$ be a Fr\'echet system of generators of $\maW_b(M)$
as $\CIb$-module, and let $\{\xi_1,
\xi_2, \ldots, \xi_\mu\}$ be the dual system,
as in Remark \ref{rem.generation2}. We shall
write $\vect{k} = (k_1, k_2, \ldots, k_\mu)$, $1 \le k_1, k_2,
\ldots k_\mu \le N$.

\begin{lemma}
\label{decomposition}
Assume that $M$ satisfies the (FFC) condition. Then any $w \in L^p(M;
T^{*\otimes \mu}M \otimes E)$ can be written as
\begin{equation*}
w \seq \sum_{\vect{k}} \xi_{k_1} \otimes \xi_{k_2} \otimes \ldots
\otimes \xi_{k_\mu} \otimes i_{Z_\mu}i_{Z_{\mu - 1}} \ldots
i_{Z_1}(w)\,.
\end{equation*}
\end{lemma}

\begin{proof}
The result follows from the relation $\omega(X) = \sum_{j=1}^N \xi_j(X)
\omega(Z_j)$ (see Remark \ref{rem.generation1}), the definition of $i_{Z_j}$,
which shows that $\omega (Z_j)=\langle \omega ,Z_j\rangle =i_{Z_j}(\omega )$, and the expression
$(\xi _i\otimes \xi _j)(X \otimes Y)  =\xi _i(X)\xi _j(Y)$
 of the tensor product of two $1$-forms $\xi_i(X)$ and $\xi_j(Y)$.
\end{proof}

The following result gives several alternative description of Sobolev
spaces $W_\nabla^{s, p}(M; E)$, $s \in \NN$, $1 \le p \le \infty$, in
terms of vector fields similar to \cite[Proposition 3.2]{GN17}, where
part of this result was proved for manifolds with bounded geometry.
Let us record the following easy lemma.
\begin{lemma}\label{lemma.finite}
Assume that $M$ satisfies the (FFC) condition. Then $w \in L^p(M;
T^{*\otimes \mu}M \otimes E)$ if, and only if, $a w \in L^p(M; E)$ for
all $a \in \CIb(M; \Hom(T^{*\otimes \mu}M \otimes E; E))$.
\end{lemma}

\begin{proof}
If $w \in L^p(M; T^{*\otimes \mu}M \otimes E)$ and $a \in \CIb(M;
\Hom(T^{*\otimes \mu}M \otimes E; E))$, then we have already seen that
$aw \in L^p(M; E)$.

Conversely, let $Z_1, Z_2, \ldots, Z_N \in \maW_b(M) := W^{\infty,
  \infty}(M; TM)$ be a Fr\'echet system of generators of $\maW_b(M)$
as $\CIb$-module, as in Remark \ref{rem.generation2}. Let $\{\xi_1,
\xi_2, \ldots, \xi_\mu\}$ be the dual system, again as in that
remark. Let $\vect{k} = (k_1, k_2, \ldots, k_\mu)$, $1 \le k_1, k_2,
\ldots k_\mu \le N$.  We let $a$ range through the composition of
contractions $i_{Z_\mu}i_{Z_{\mu - 1}} \ldots i_{Z_1}$, which recovers
$a$. Then Lemma \ref{decomposition} implies that
\begin{equation*}
w \seq \sum_{\vect{k}} \xi_{k_1} \otimes \xi_{k_2} \otimes \ldots
\otimes \xi_{k_\mu} \otimes i_{Z_\mu}i_{Z_{\mu - 1}} \ldots
i_{Z_1}(w)\,.
\end{equation*}
If all $aw \in L^p(M; E)$, then by taking $a := \xi_{k_1} \otimes
  \xi_{k_2} \otimes \ldots \otimes \xi_{k_\mu} \otimes
  i_{Z_\mu}i_{Z_{\mu - 1}} \ldots i_{Z_1}(w)$, we get that all
\begin{equation*}
\xi_{k_1} \otimes \xi_{k_2} \otimes \ldots \otimes \xi_{k_\mu}
\otimes i_{Z_\mu}i_{Z_{\mu - 1}} \ldots i_{Z_1}(w) \in L^p(M;
T^{*\otimes \mu}M \otimes E) \,,
\end{equation*}
and hence $w \in L^p(M; T^{*\otimes \mu}M \otimes E)$.
\end{proof}

The following type of descriptions is often used in the setting
of weighted Sobolev spaces, see, for instance
\cite{ammannNistorWReg, BNZ3D1, BMNZ, CDN12, daugeBook,
Kondratiev67, KMR, NP}.

\begin{proposition}\label{prop.alternative}
Let $Z_1, Z_2, \ldots, Z_N \in \maW_b(M) := W^{\infty, \infty}(M; TM)$
be a Fr\'echet system of generators of $\maW_b(M)$ as $\CIb$-module,
as in Remark \ref{rem.generation2}, if $M$ satisfies the (FCC)
condition. Let $s \in \NN$ and $1 \le p \le \infty$. Then the
following spaces all coincide with $W_\nabla^{s, p}(M; E)$ under the
listed additional conditions:
\begin{enumerate}[(i)]
\item $W_\nabla^{s, p}(M; E) \seq \{ u \, \vert\ P u \in L^p(M; E)\,,
  \ \forall P \in \Diff_b^s(M; E, F)\}$.

\item $W_\nabla^{s, p}(M; E) \seq \{ u \, \vert\ P u \in L^p(M; E)\,,
  \ \forall P \in \Diff_b^s(M; E)\}$, provided that $M$ satisfies the
  (FFC) condition.

\item $W^{\ell, p}_{\nabla }(M; E) \seq \{\, u \, \vert
  \ \nabla^E_{Z_{k_1}} \nabla^E_{Z_{k_2}} \ldots \nabla^E_{Z_{k_j}} u
  \in L^p(M; E),\ j \leq \ell, \ 1 \le k_1, k_2, \ldots, k_j \leq N \,
  \}$, provided that $M$ satisfies the (FFC) condition.

\item $W^{\ell, p}_{\nabla }(M; E) \seq \{\, u \, \vert
  \ \nabla^E_{Z_{k_1}} \nabla^E_{Z_{k_2}} \ldots \nabla^E_{Z_{k_j}} u
  \in L^p(M; E),\ j \leq \ell, \ 1 \le k_1 \le k_2 \le \ldots \le k_j
  \leq N \, \}$ provided that $M$ satisfies the (FFC) condition and
  $E$ has totally bounded curvature.
\end{enumerate}
\end{proposition}

\begin{proof}
The first characterization of the Sobolev spaces $W_\nabla^{\mu, p}(M;
E)$ follows from the definition of $\nabla$-Sobolev spaces (Definition
\ref{def.Sobolev}) and $\nabla$-differential operators (Definition
\ref{def.diff.op}), since it is enough to take $P$ among the monomials
$\nabla^j$, $0 \le j \le \mu$. (Indeed, any $P \in \Diff_b^s(M; E, F)$
has the form $P = \sum_{j=0}^{\mu} a_j \nabla^j $ with $a_{j} \in
\CIb(M; \Hom(E; F))$.)

The second point is similar. Indeed, we have $w \in L^p(M; T^{*\otimes
  k}M \otimes E$ if, and only if, $a w \in L^p(M; T^{*\otimes k}M
\otimes E$ for all $a \in \CIb(M; \Hom(T^{*\otimes k}M \otimes E)$. By
applying this observation to $w := \nabla^j u$, $j \le \mu$, and using
Lemma \ref{lemma.finite}, the definitions of $\nabla$-Sobolev spaces
and $\nabla$-differential operators, we obtain the result.

The third and fourth points are also similar. They follow by combining
(i) with Lemma \ref{lemma.finite} and Proposition
\ref{prop.equality.mixed} (for (iii)), respectively
Corollary \ref{cor.order2} for (iv).
\end{proof}


\subsection{Bidifferential operators and Dirichlet forms}
Let $M_1$ and $M_2$ be two topological spaces and let
\begin{equation*}
  \pi_j : M_1 \times M_2 \to M_j
\end{equation*}
be the projection onto the $j$th component, $j = 1, 2$. For any two
real or complex vector bundles $E_j \to M_j$, we let
\begin{equation*}
    E_1 \boxtimes E_2 \ede \pi_1^*(E_1) \otimes \pi_2^*(E_2) \to M_1
    \times M_2
\end{equation*}
be the {\em external tensor product} of $E_1$ and $E_2$. It is a
vector bundle on $M_1 \times M_2$. More concretely, if $x_j \in M_j$
and $E_{j,x_j}$ is the fiber of $E_j$ above $x_j$, then the fiber of
$E_1 \boxtimes E_2$ above $(x_1, x_2)$ is $E_{1, x_1} \otimes E_{2,
  x_2}$. If $M_1 = M_2 = M$, in which case we shall always regard $M$
as being diagonally embedded in $M \times M$, then, of course,
\begin{equation*}
    E_1 \boxtimes E_2 \vert_{M} \seq E_1 \otimes E_2 \,.
\end{equation*}

\begin{remark}\label{rem.splitting}
If $M_j$ are smooth manifolds and $E_j \to M_j$ are smooth vector
bundles endowed with connections, $j = 1, 2$, then $E_1 \boxtimes
E_2$ is endowed with the canonically induced connection from
$\pi_j^*(E_j)$ (which acts trivially on the fiber $M_{3-j}$ of $\pi_j
: M_1 \times M_2 \to M_j$). Let us take a closer look at this induced
connection on $E_1 \boxtimes E_2$.  We first notice that we have a
canonical isomorphism
\begin{equation*}
    T(M_1 \times M_2) \simeq TM_1 \times TM_2 \simeq \pi_1^*(TM_1)
    \oplus \pi_2^*(TM_2)\,.
\end{equation*}
Let $p_j : T^*(M_1 \times M_2) \to \pi_j^*(T^*M_j)$ be the induced
projections and let
\begin{equation*}
  \nabla^{E_1 \boxtimes E_2} \seq \nabla_1 + \nabla_2 \ \mbox{ where }
  \ \nabla_j \ede (p_j \otimes id_{E_1 \boxtimes E_2}) \circ
  \nabla^{E_1 \boxtimes E_2}\,.
\end{equation*}
If $u_j \in \CIc(M_j; E_j)$, $j = 1, 2$, we let $v := u_1 \otimes u_2
\in \CIc(M_1 \times M_2 ; E_1 \boxtimes E_2)$, that is, $v(x_1, x_2)
\ede (u_1 \otimes u_2) (x_1, x_2) \ede u(x_1) \otimes u_2(x_2).$ We
then obtain
\begin{equation*}
    \nabla_1 v \seq \nabla u_1 \otimes u_2\,,\ \nabla_2 v \seq u_1
    \otimes \nabla u_2\,,\, \mbox{ and hence }\, \nabla_1 \nabla_2
    \seq \nabla_2 \nabla_1\,.
\end{equation*}
\end{remark}

Let $V$ be a complex vector space or vector bundle and let
$\overline{V}$ denote the conjugate space to $V$. More precisely,
$\overline{V} = V$ as additive groups, but with the external
multiplication ${\overline z}\, {\overline v}= \overline{z v}$, $z \in
\CC$, $v \in V$, where $\overline{v}$ denotes the image in
$\overline{V}$ of an element $v \in V$.


\begin{definition}\label{def.bidifferential}
Let $E, F \to M$ be two smooth vector bundles endowed with
connections. A {\em $\nabla$--bidifferential operator} on $(E, F)$ is a linear map
\begin{equation*}
\mfkb^{\nabla} : \CIc(M \times M ; E \boxtimes \overline{F}) \to \CIc(M)
\end{equation*}
of the form $\mfkb^{\nabla} v = P v\vert_{M}$, where $P : \CIc(M
\times M; E \boxtimes \overline{F}) \to \CIc(M \times M)$ is a
$\nabla$-differential operator with smooth
coefficients.

If we can choose $P$ to have $\CIb$-coefficients (that
is in $\Diff_b(M \times M; E \boxtimes \overline{F}, \CC)$), then we
say that $\mfkb^{\nabla}$ has $\CIb$-coefficients as well.
\end{definition}

Bidifferential operators appeared also in the framework of deformation
quantization. Let us obtain a more explicit form of the
$\nabla$-bidifferential operators.

\begin{remark}\label{rem.can.form}
We use the notation of Definition \ref{def.bidifferential} and let
$\pi_ j: M \times M \to M$, $j = 1, 2$, be the two projections.  Let
$\nabla^{E \boxtimes \overline{F}} \seq \nabla_1 + \nabla_2$ be the
decomposition of Remark \ref{rem.splitting}. Thus, if $u \in \CIc(M;
E)$, $w \in \CIc(M; F)$, and $v := u \otimes \overline{w} \in \CIc(M
\times M ; E \boxtimes \overline{F})$, then $\nabla_1 v = \nabla u
\otimes \overline{w}$, $\nabla_2 v = u \otimes \nabla
\overline{w}$. Therefore
\begin{multline*}
    P v \ede a \cdot \nabla^{tot}v \seq \sum_{j=0}^\mu a^{[j]}
    \nabla^j v \seq \sum_{i + j \le \mu } \tilde \mfka_{ij} \big(
    \nabla_1^i \nabla_2^j v \big )\\
    \seq \sum_{i + j \le \mu} \tilde \mfka_{ij} \big[ (\nabla^iu)
      \otimes (\nabla^j \overline{w}) \big] \in \CIc(M \times M) \,,
\end{multline*}
where each $\tilde \mfka_{ij} \in \CI\Big(M \times M; \big[ (T^{*
    \otimes i} M \otimes E) \boxtimes (T^{* \otimes j} M \otimes
  \overline{F})\big]^\prime \Big )$ is obtained in a canonical
(linear) way from $a \in \CI\big (M \times M; \maF_\mu^{M \times M}(E
\boxtimes \overline{F})^\prime \big)$. In particular, if $\tilde a \in
\CIb$, then all $\tilde a_{ij} \in \CIb$ as well and their bounds are
controlled by the bounds for $a$. Let us assume also that $F$ is
endowed with a Hermitian metric $(\,,\,)_F$, regarded as a bilinear
form on $F \otimes \overline{F}$. We let $(\,,\,)_{T^{*\otimes j}M
  \otimes F}$ be the corresponding hermitian form (i.e. bilinear form
on $T^{*\otimes j}M \otimes F \otimes T^{*\otimes j}M \otimes
\overline{F}$). Then there exist unique
\begin{equation*}
    \mfka_{ij} \in \CI \big(M; \Hom(T^{*\otimes i}M \otimes E;
    T^{*\otimes j}M \otimes F)\big)
\end{equation*}
such that $(\mfka_{ij} \xi, \eta) = \tilde a_{ij} \vert_M (\xi \otimes
\overline{\eta})$ and hence we have the following {\em canonical form}
for $\mfkb^\nabla$:
\begin{equation*}
    \mfkb^{\nabla} v(x) \seq \mfkb^{\nabla} (u \otimes
    \overline{w})(x) \seq \sum_{i + j \le \mu} \big (\mfka_{ij}(x)
    \nabla^iu(x), \nabla^j w(x) \big)_{T^{*\otimes j} M \otimes F}\,.
\end{equation*}
If $\mfka_{ij} = 0$ for $i > m$ or $j > m$, we shall say that
$\mfkb^\nabla$ has order $\le 2m$.
\end{remark}

\begin{lemma}\label{lemma.generation.bidiff}
Let $P \in \Diff^m(M; E, G)$ and $Q \in \Diff^m(M; F, G)$, where $G$
is a Hermitian vector bundle. Then $\mfkb(u \otimes \overline{w}):=
(Pu, Qw)_G$ is a $\nabla$-bidifferential operator $\CIc(M \times M ; E
\boxtimes \overline{F}) \to \CIc(M)$ of order $\le 2m$ with smooth
coefficients. If $P$ and $Q$ have $\CIb$-coefficients, then $\mfkb$
will also have $\CIb$-coefficients.
\end{lemma}

\begin{proof} We can assume $P = a\nabla^i$ and $Q = b \nabla^j$, by
  linearity and by the definition of $\nabla$-differential operators.
\end{proof}

Recall that $\dvol$ denotes the volume form on $M$ associated to the
metric.

\begin{definition}\label{def.can.form}
The sesquilinear form
\begin{equation*}
    B_\mfkb^{\nabla}(u, w) \ede \int_M \mfkb^{\nabla} (u \otimes
    \overline{w}) \dvol\
\end{equation*}
is called the {\em Dirichlet form} associated to $\mfkb^{\nabla}$. It
has $\CIb$-coefficients if $\mfkb^{\nabla}$ has and it has the same
order as $\mfkb^{\nabla}$. The induced map $P_\mfkb^\nabla : \CIc(M;
E) \to \CIc(M; F)^*$,
\begin{equation*}
    \langle P_\mfkb^\nabla u, \overline{w} \rangle \seq
    B_\mfkb^{\nabla}(u, w)\,, \quad u \in \CIc(M; E), w \in \CIc(M;
    F)\,,
\end{equation*}
is called {\em the $\nabla$-differential operator in divergence form}
associated to $\mfkb^\nabla$ (or to $B_\mfkb^\nabla$). If
$\mfka_{ij} = 0$ if $i > m$ or $j > m$, then we shall say that $
P_\mfkb^\nabla$ has order $\le 2m$.
\end{definition}

We shall see in the last part of this section that $P_\mfkb^\nabla$ is indeed a
$\nabla$-differential operator. We shall continue to use the notation
of Remark \ref{rem.can.form}. In particular, $\mfkb^{\nabla}$ will be
a $\nabla$-bidifferential operator.

\begin{lemma}\label{lemma.B.continuous}
Let $m \in \NN$ be such that $\mfka_{ij} \in \CIb(M; \Hom(T^{*\otimes
  i}M \otimes E; T^{*\otimes j}M \otimes F))$, for all $i$ and $j$ and
$a_{ij} = 0$ if $i > m$ or $j > m$. Then $B_\mfkb^{\nabla}$ extends to
a continuous, {\em sesquilinear} map
\begin{equation*}
    B_\mfkb^{\nabla} : H_\nabla^m(M; E) \times H_{\nabla}^m(M; F) \to
    \CC.
\end{equation*}
Similarly, $P_\mfkb^\nabla$ extends to a continuous map
\begin{equation*}
    P_\mfkb^{\nabla} : H_\nabla^m(M; E) \to H_{\nabla}^{m}(M; F)^*.
\end{equation*}
\end{lemma}

\begin{remark}\label{rem.vector.space}
It is clear that all of the following {sets} are vector spaces:
\begin{enumerate}
  \item The set $\BiDiff_{\nabla}^{2m}(M; E, F)$ of
    $\nabla$-bidifferential operators $\mfkb^\nabla : \CIc(M \times M;
    E \boxtimes \overline{F}) \to \CIc(M)$ of order $\le 2m$ (with
    smooth coefficients).

  \item The set $\BiDiff_{b, \nabla}^{2m}(M; E, F)$ of
    $\nabla$-bidifferential operators $\mfkb^\nabla \in
    \BiDiff^{\mu}(M; E, F)$ with $\CIb$-coefficients.

  \item The set of Dirichlet forms $B_\mfkb^{\nabla}$ (with smooth
    coefficients) associated to $\mfkb^\nabla \in \BiDiff^{2m}(M; E,
    F)$.

  \item The set of Dirichlet forms $B_\mfkb^{\nabla}$ associated to
    $\mfkb^\nabla \in \BiDiff_b^{\mu}(M; E, F)$ (thus with
    $\CIb$-coefficients).

  \item The set of order $\le 2m$ differential operators
    $P_\mfkb^{\nabla}$ in divergence form (with smooth coefficients)
    associated to $\mfkb^\nabla \in \BiDiff^{\mu}(M; E, F)$.

  \item The set of order $\le 2m$ differential operators
    $P_\mfkb^{\nabla}$ in divergence form associated to $\mfkb^\nabla
    \in \BiDiff_b^{\mu}(M; E, F)$ (thus with $\CIb$-coefficients).
\end{enumerate}
\end{remark}

The coefficients $\mfka_{ij}$ in the canonical form for $\mfkb^\nabla$
(see Remark \ref{rem.splitting}) are not unique (except for $i + j \le
1$). So when we say that one of the above objects has
$\CIb$-coefficients, we mean that we can choose the coefficients
$\mfka_{ij}$ to be in $\CIb$.

\begin{proposition}\label{prop.mixed.bidiff}
A linear map $\mfkb^{mix} : \CIc(M \times M ; E \boxtimes
\overline{F}) \to \CIc(M)$ is a {\em $\nabla$-bidifferential
  operator of order $\le 2m$} if, and only if, it is a linear
combination of maps of the form $u \otimes \overline{w} \to (Pu,
Qw)_G$, where $P \in \Diff^m(M; E, G)$ and $Q \Diff^m(M; F, G)$.
If $\mfkb$ has $\CIb$-coefficients, then we can choose
$P$ and $Q$ to also have $\CIb$-coefficients.
\end{proposition}

\begin{proof}
This follows from definitions.
\end{proof}

\begin{proposition}\label{prop.ind.A2}
Let $A \in \CI(M; \End(E))$, $A^* = -A$, and $\widetilde \nabla =
\nabla + A$. Let also $p \in [1, \infty]$ and $\mu \in \NN$. We then
have the following:
\begin{enumerate}[(i)]

\item $\BiDiff_{\nabla}^\mu(M; E, F) = \BiDiff_{\widetilde
  \nabla}^\mu(M; E, F)$.

\item If $A \in \CIb(M; \End(E))$, then $\BiDiff_{b, \nabla}^\mu(M; E,
  F) = \BiDiff_{b, \widetilde \nabla}^\mu(M; E, F)$.
\end{enumerate}
\end{proposition}

\begin{proof}
This follows with arguments similar to those for Proposition \ref{prop.inclusion.mixed1}.
\end{proof}

Let us now introduce {\em mixed $\nabla$-bidifferential operators}.
Recall that $\Diff^m(M; E, G)$ denotes the set of mixed differential
operators $\CIc(M; F) \to \CIc(M; G)$.

\begin{definition}\label{def.mixed.bidiff}
A linear map $\mfkb^{mix} : \CIc(M \times M ; E \boxtimes
\overline{F}) \to \CIc(M)$ is a {\em mixed bidifferential
operator of order $\le 2m$} if
it is a linear combination of maps of the form $u \otimes \overline{w} \to (Pu,
Qw)_G$, where $P \in \widetilde \Diff^m(M; E, G)$ and $Q \in
\widetilde \Diff^m(M; F, G)$.  If $P$ and $Q$ have
$\CIb$-coefficients, then we shall say that $\mfkb$ will also have
$\CIb$-coefficients.
\end{definition}

Let us assume now (FFC) and derive then the equality of the space of
$\nabla$ bidifferential operators and that of mixed bidifferential
operators, in analogy with the corresponding result for differential
operators.

\begin{proposition}
\label{prop.equality.mixed2}
Let $m \in \NN$ and assume $M$ satisfies (FFC).
\begin{enumerate}[(i)]
\item We have $\widetilde \BiDiff^{2m}(M; E, F) = \BiDiff^{2m}(M; E,
  F)$ and, similarly, $\widetilde \BiDiff_b^{2m}(M; E, F) =
  \BiDiff_b^{2m}(M; E, F)$.

\item Let $Z_1, Z_2, \ldots, Z_N \in \maW_b(M)$ be Fr\'echet systems
of generators for $\maW_b(M)$ of Remark \ref{rem.generation2}, then
$\widetilde \Diff^{2m}(M; E, F)$ is linearly generated by
sesqui-linear maps of the form
\begin{equation*}
u \otimes \overline{w} \to (a \nabla_{X_1}\nabla_{X_2} \ldots
\nabla_{X_r}, b \nabla_{X_{r+1}}\nabla_{X_{r+2}} \ldots
\nabla_{X_{r+s}})_G\,,
\end{equation*}
where $r \le m$, $X_1, X_2, \ldots, X_{r+s} \in \{Z_1, Z_2, \ldots,
Z_N\}$, $a \in \CI(M; \Hom(E, G))$, $b \in \CI(M; \Hom(F, G))$, and
$G$ is an auxiliary vector bundle.

\item The analogous result holds for $\widetilde \BiDiff_b^{2m}(M; E,
  F)$ with $a \in \CIb(M; \Hom(E, G))$ and $b \in \CIb(M; \Hom(F,
  G))$.
\end{enumerate}
\end{proposition}

\begin{proof}
This follows from Proposition \ref{prop.inclusion.mixed1}
and Lemma \ref{lemma.generation.bidiff}. 
\end{proof}

\begin{proposition}\label{prop.order2}
Assume $M$ satisfies the (FFC) condition and let $(Z_j)$, $1 \le j \le
N$, be a Fr\'echet generating system for $\maW_b(M)$.
\begin{enumerate}[(i)]
\item $\BiDiff^{2m}(M; E, F)$ is linearly generated by sesquilinear
  maps of the form $u \otimes \overline{w} \to \big(a \nabla^E_{Z_{k_1}}
  \ldots \nabla^E_{Z_{k_r}} u, b \nabla^E_{Z_{j_1}} \ldots
  \nabla^E_{Z_{j_s}}\big)_G$, where $1 \le k_1 \le k_2 \le \ldots \le k_r
  \le N$, $1 \le j_1 \le j_2 \le \ldots \le j_s \le N$, $r, s \le m$,
  and $a \in \CI(M; \Hom(E; G))$ and $b \in \CI(M; \Hom(F;
  G))$.\smallskip

  \hspace*{-1.3cm} Let us assume also that $E \to M$ and $F \to M$
  have totally bounded curvature. Then
\smallskip

\item $\BiDiff_b^{2m}(M; E, F)$ is linearly generated by sesquilinear
  maps $\mfkb$ of the form $\mfkb(u \otimes \overline{w}) = \big(a
  \nabla^E_{Z_{k_1}} \ldots \nabla^E_{Z_{k_r}} u, b \nabla^E_{Z_{j_1}}
  \ldots \nabla^E_{Z_{j_s}}\big)_G$, where $1 \le k_1 \le k_2 \le \ldots
  \le k_r \le N$, $1 \le j_1 \le j_2 \le \ldots \le j_s\le N$, $r, s
  \le m$, and $a \in \CIb(M; \Hom(E; G))$ and $b \in \CIb(M; \Hom(F;
  G))$.
\end{enumerate}
\end{proposition}

\begin{proof}
This follows from Corollary \ref{cor.order2} and Lemma
\ref{lemma.generation.bidiff}.
\end{proof}

We conclude with the following proposition.

\begin{proposition}\label{prop.divergence.form}
Assume that $M$ satisfies the (FFC) condition and let $m\in \mathbb N$ and $\mfkb \in
\BiDiff^{2m}(M; E, F)$. Then the restriction $P$ of $P_\mfkb^\nabla$,
\begin{equation*}
    P_\mfkb^{\nabla} : H_\nabla^m(M; E) \to H_{\nabla}^{-m}(M; F)\,,
\end{equation*}
obtained from the restriction $H_{\nabla}^{m}(M; F)^m \to
H_{\nabla}^{-m}(M; F)$, is a $\nabla$-differential operator of
order $\le 2m$, that is $P \in \Diff_b^{2m}(M; E, F)$.
\end{proposition}

\begin{proof} By linearity, we can assume that $\mfkb(u \otimes
  \overline{w}) = (a \nabla^i u, \nabla^j w)_{T^{*\otimes j}M \otimes
    F}$, where $a \in \CIb(M; \Hom(T^{*\otimes i}M \otimes E;
  T^{*\otimes j}M \otimes F)$. Then $P = (\nabla^j)^* a \nabla^i \in
  \Diff_b^{2m}(M; E, F)$ by Proposition \ref{mapping-properties}(ii).
\end{proof}

\begin{corollary}
\label{divergence-form-weight}
Assume that $M$ satisfies (FFC) and let $m\in \mathbb N$ and $\mfkb \in \BiDiff^{2m}(M; E, F)$
with respect to the metric $g$ of $M$. Let $\rho , f_0: M \to (0, \infty)$
be admissible weights with respect to the metric $g_0=\rho ^{-2}g$.
Then the restriction $P$ of $P_\mfkb^\nabla$ to weighted Sobolev spaces,
\begin{equation*}
P_\mfkb^{\nabla} : f_0 H_{\nabla^{LC}, \rho}^m(M; E) \to \frac{1}{f_0}H_{\nabla^{LC}, \rho}^{-m}(M; F)\,,
\end{equation*}
obtained from the restriction $f_0H_{\nabla^{LC}, \rho}^{m}(M; F)^m \to
\frac{1}{f_0}H_{\nabla^{LC}, \rho }^{-m}(M; F)$, is a $\nabla $-differential operator of
order $\le 2m$, that is $P \in \Diff_b^{2m}(M; E, F)$.
\end{corollary}

\begin{proof}
The result follows from Proposition \ref{prop.divergence.form} and
Corollary \ref{diff-oper-weight}.
\end{proof}

\def\cprime{$'$}

\end{document}